\documentclass[11pt,a4paper]{amsart}
\usepackage{amsmath, amsthm, amsfonts, amssymb}
\usepackage[bookmarks=false]{hyperref}

\usepackage{color} 

\numberwithin{equation}{section}

\newtheorem*{maintheorem}{Main Theorem}

\newtheorem{thm}{Theorem}[section]
\newtheorem{lem}[thm]{Lemma}

\newtheorem{prop}[thm]{Proposition}

\theoremstyle{definition}
\newtheorem{rem}[thm]{Remark}

\newtheorem{nota}[thm]{Notation}

\newtheorem*{definition}{Definition}
\newtheorem*{claim}{Claim}

\newcommand{\R}{\mathbf{R}}
\newcommand{\C}{\mathbf{C}}

\newcommand{\F}{\mathbf{F}}

\newcommand{\N}{\mathbf{N}}

\newcommand{\Cdec}{\mathcal C_{\text{\rm anti-free}}}

\newcommand{\Ad}{\operatorname{Ad}}
\newcommand{\id}{\text{\rm id}}

\newcommand{\Aut}{\operatorname{Aut}}
\newcommand{\rL}{\mathord{\text{\rm L}}}

\newcommand{\rE}{\mathord{\text{\rm E}}}
\newcommand{\Sp}{\mathord{\text{\rm Sp}}}

\newcommand{\tr}{\mathord{\text{\rm tr}}}
\newcommand{\core}{\mathord{\text{\rm c}}}
\newcommand{\Proj}{\mathord{\text{\rm Proj}}_{\text{\rm f}}}

\newcommand{\Tr}{\mathord{\text{\rm Tr}}}

\newcommand{\Ball}{\mathord{\text{\rm Ball}}}

\newcommand{\ovt}{\mathbin{\overline{\otimes}}}

\newcommand{\dpr}{^{\prime\prime}}

\begin{document}

\title[Rigidity of free product von Neumann algebras]{Rigidity of free product von Neumann algebras}

\begin{abstract}
Let $I$ be any nonempty set and $(M_i, \varphi_i)_{i \in I}$ any family of nonamenable factors, endowed with arbitrary faithful normal states, that belong to a large class $\Cdec$ of (possibly type ${\rm III}$) von Neumann algebras including all nonprime factors, all nonfull factors and all factors possessing Cartan subalgebras. For the free product $(M, \varphi) = \ast_{i \in I} (M_i, \varphi_i)$, we show that the free product von Neumann algebra $M$ retains the cardinality $|I|$ and each nonamenable factor $M_i$ up to stably inner conjugacy, after permutation of the indices. Our main theorem unifies all previous Kurosh-type rigidity results for free product type ${\rm II_1}$ factors and is new for free product type ${\rm III}$ factors. It moreover provides new rigidity phenomena for type ${\rm III}$ factors.
\end{abstract}

\author{Cyril Houdayer}
 \address{Laboratoire de Math\'ematiques d'Orsay\\ Universit\'e Paris-Sud\\ CNRS\\ Universit\'e Paris-Saclay\\ 91405 Orsay\\ France}
\email{cyril.houdayer@math.u-psud.fr}

\thanks{CH is supported by ANR grant NEUMANN and ERC Starting Grant GAN 637601.}

\author{Yoshimichi Ueda}
\address{Graduate School of Mathematics \\ Kyushu University \\ Fukuoka, 819-0395 \\ Japan}
\email{ueda@math.kyushu-u.ac.jp}
\thanks{YU is supported by Grant-in-Aid for Scientific Research (C) 24540214.}


\maketitle

\section{Introduction and statement of the Main Theorem}

In his seminal article \cite{Oz04}, Ozawa obtained the first Kurosh-type rigidity results for free product type ${\rm II_1}$ factors. Among other things, he showed that whenever $m \geq 1$ and $M_1, \dots, M_m$ are weakly exact nonamenable {\em nonprime} type ${\rm II_1}$ factors, the tracial free product von Neumann algebra $M_1 \ast \cdots \ast M_m$ retains the integer $m$ and each factor $M_i$ up to inner conjugacy, after permutation of the indices. Ozawa's approach to Kurosh-type rigidity for ${\rm II_1}$ factors was based on a combination of his C$^*$-algebraic techniques \cite{Oz03} and of Popa's intertwining techniques \cite{Po01, Po03} (see also \cite{OP03}). Shortly after, using Popa's deformation/rigidity theory, Ioana--Peterson--Popa obtained in \cite{IPP05} Kurosh-type rigidity results for tracial free products of {\em weakly rigid} type ${\rm II_1}$ factors, that is, ${\rm II_1}$ factors possessing regular diffuse von Neumann subalgebras with relative property (T) in the sense of \cite{Po01}. These Kurosh-type rigidity results for ${\rm II_1}$ factors were then unified and further generalized by Peterson in \cite{Pe06}, using his $\rL^2$-rigidity techniques, to cover tracial free products of nonamenable $\rL^2$-{\em rigid} type ${\rm II_1}$ factors. In \cite{As09}, Asher extended Ozawa's original result \cite{Oz04} to free products of weakly exact nonamenable nonprime type ${\rm II_1}$ factors with respect to nontracial states.

Regarding the structure of free product von Neumann algebras, the questions of factoriality, type classification and fullness for arbitrary free product von Neumann algebras were recently completely solved by Ueda in \cite{Ue10}. For any free product $(M, \varphi) = (M_1, \varphi_1) \ast (M_2, \varphi_2)$ with $\dim_\C M_i \geq 2$ and $(\dim_\C M_1,\dim_\C M_2) \neq (2,2)$, the free product von Neumann algebra $M$ splits as a direct sum $M = M_c \oplus M_d$ where $M_c$ is a full factor of type ${\rm II_1}$ or of type ${\rm III_\lambda}$ (with $0 < \lambda \leq 1$) and $M_d = 0$ or $M_d$ is a multimatrix algebra. Moreover,  Chifan--Houdayer showed in \cite{CH08} (see also \cite{Ue10}) that $M_c$ is always a {\em prime} factor (see Peterson \cite{Pe06} for the previous work in the tracial case) and Boutonnet--Houdayer--Raum showed in \cite{BHR12} that $M_c$ has no Cartan subalgebra (see Ioana \cite{Io12} for the the previous work in the tracial case). Very recently, in our joint work \cite{HU15}, we completely settled the questions of maximal amenability and maximal property Gamma of the inclusion $M_1 \subset M$ in arbitrary free product von Neumann algebras. In view of these recent structural results obtained in full generality, it is thus natural to seek for Kurosh-type rigidity results for {\em arbitrary} free product von Neumann algebras.

In this paper, we unify and generalize all the previous Kurosh-type rigidity results to {\em arbitrary} free products $(M, \varphi) = \ast_{i \in I} (M_i, \varphi_i)$ over {\em arbitrary} index sets $I$, where all $M_i$ are nonamenable factors that belong to a large class of (possibly type ${\rm III}$) factors that we call {\em anti-freely decomposable}. In order to state our Main Theorem, we will use the following terminology.

\begin{definition}
We will say that a nonamenable factor $M$ with separable predual is {\em anti-freely decomposable} if at least one of the following conditions holds. 
\begin{itemize}
\item[(i)] $M$ is {\em not prime}, that is, $M = M_1 \ovt M_2$ where $M_1$ and $M_2$ are diffuse factors (e.g.\ $M$ is McDuff).
\item[(ii)] $M$ has {\em property Gamma}, that is, the central sequence algebra $M' \cap M^\omega$ is diffuse for some nonprincipal ultrafilter $\omega \in \beta(\N) \setminus \N$ (e.g.\ $M$ is of type ${\rm III_0}$; see \cite[Proposition 3.9]{Co74}).
\item[(iii)]  $M$ possesses an amenable finite von Neumann subalgebra $A$ with expectation such that $A' \cap M = \mathcal Z(A)$ and $\mathcal N_M(A)\dpr = M$ (e.g.\ $M$ possesses a Cartan subalgebra).
\item [(iv)] $M$ is a ${\rm II_1}$ factor that possesses a regular diffuse von Neumann subalgebra with  relative property (T) in the sense of \cite[Definition 4.2.1]{Po01} (e.g.\ $M$ is a ${\rm II_1}$ factor with property (T) \cite{CJ85}).
\end{itemize}
\end{definition}
We will denote by $\Cdec$ the class of nonamenable factors with separable predual that are anti-freely decomposable in the sense of the above definition. 

Recall that the {\em Kurosh isomorphism theorem} for discrete groups (see e.g.\ \cite[pp 105]{CM82}) says that any discrete group can uniquely (up to permutation of components) be decomposed into a free product of {\em freely indecomposable} subgroups. It is not clear at all how to capture {\em freely indecomposable von Neumann algebras} practically. However, all the known {\em general} structural results on free product von Neumann algebras suggest that $\Cdec$ is indeed a natural large class of freely indecomposable factors. Hence the Main Theorem of this paper stated below is indeed a von Neumann algebra counterpart of the Kurosh isomorphism theorem and unifies all the previous counterparts.

\begin{maintheorem}
Let $I$ and $J$ be any  nonempty sets and $(M_i)_{i \in I}$ and $(N_j)_{j \in J}$ any families of nonamenable factors in the class $\Cdec$. For each $i \in I$ and each $j \in J$, choose any faithful normal states $\varphi_i \in (M_i)_\ast$ and $\psi_j \in (N_j)_\ast$. Denote by $(M, \varphi) = \ast_{i \in I} (M_i, \varphi_i)$ and $(N, \psi) = \ast_{j \in J} (N_j, \psi_j)$ the corresponding free products.

\begin{enumerate}
\item Assume that $M$ and $N$ are isomorphic. Then $|I| = |J|$ and there exists a bijection $\alpha : I \to J$ such that $M_i$ and $N_{\alpha(i)}$ are stably isomorphic for all $i \in I$.

\item Assume that $M$ and $N$ are isomorphic and identify $M = N$. Assume moreover that $M_i$ is a type ${\rm III}$ factor for all $i \in I$. Then there exists a unique bijection $\alpha : I \to J$ such that $M_i$ and $N_{\alpha(i)}$ are inner conjugate for all $i \in I$.
\end{enumerate} 
\end{maintheorem}

Our Main Theorem is new for free products of type ${\rm III}$ factors. In that case (see item (2)), our statement is as sharp as all previous Kurosh-type rigidity results for free products of type ${\rm II_1}$ factors. We point out that for {\em tracial} free products, our Main Theorem is still new in cases (i), (ii) and (iii) when the index set $I$ is {\em infinite} (compare with \cite{Oz04, Pe06, Io12}).

We now briefly explain the strategy of the proof of the Main Theorem. We refer to Sections \ref{structure} and \ref{proof} for further details. As we will see, the proof builds upon the tools and techniques we developed in our previous work \cite{HU15} on the {\em asymptotic structure} of free product von Neumann algebras. Using the very recent generalization of Popa's intertwining techniques in \cite[\S4]{HI15}, it suffices, modulo some technical things, to prove the existence of a bijection $\alpha : I \to J$ such that $M_i \preceq_M N_{\alpha(i)}$ and $N_{\alpha(i)} \preceq_M M_i$ for all $i \in I$. To simplify the discussion, fix $i \in I$. We need to show that there exists $j \in J$ such that $M_i \preceq_M N_j$.

Firstly, assume that $M_i$ is in case (i) or (ii).  Exploiting the anti-free decomposability property of $M_i$ (in case (i)) and a new characterization of property Gamma for arbitrary von Neumann algebras (in case (ii)) (see Theorem \ref{theorem-structure-gamma}) together with various technical results from our previous work \cite{HU15}, it suffices to prove that for a well-chosen diffuse abelian subalgebra $A \subset M_i$ with expectation whose relative commutant $A' \cap M_i$ is nonamenable, there exists $j \in J$ such that $A \preceq_M N_j$. This is achieved in Theorem \ref{thm-spectral-gap-general-AFP} by using a combination of Popa's spectral gap argument \cite{Po06} together with Connes--Takesaki's  structure theory for type ${\rm III}$ von Neumann algebras \cite{Co72, Ta03} and Houdayer--Isono's recent intertwining theorem \cite{HI15}. Secondly, assume that $M_i$ is in case (iii). Then it suffices  again to prove that there exists $j \in J$ such that $A \preceq_M N_j$. The proof is slightly more involved (see Theorem \ref{thm-normalizer-general-AFP}) and relies on Vaes's recent dichotomy result for normalizers inside tracial amalgamated free product von Neumann algebras \cite{Va13} (improving Ioana's previous result \cite{Io12} and involving Popa--Vaes's striking dichotomy result \cite{PV11}) instead of Popa's spectral gap argument \cite{Po06} (see Appendix \ref{appendix}). Thirdly, assume that $M_i$ is in case (iv). Then it suffices to prove that there exists $j \in J$ such that $A \preceq_M N_j$ where $A \subset M_i$ is a diffuse regular subalgebra with relative property (T). This is achieved in Theorem \ref{thm-relative-property-T-general-AFP} by reconstructing \cite[Theorem 4.3]{IPP05} in the semifinite setting. 

In Section \ref{further-results}, we prove further new results regarding the structure of free product von Neumann algebras. In particular, we obtain a complete characterization of {\em solidity} \cite{Oz03} for free products with respect to arbitrary faithful normal states and over arbitrary index sets.

\subsection*{Acknowledgments} The first named author is grateful to Sven Raum for allowing him to include in this paper their joint result Theorem \ref{theorem-structure-gamma} obtained through their recent work \cite{HR14}. He also  warmly thanks Adrian Ioana for sharing his ideas with him and for thought-provoking discussions that led to Proposition \ref{proposition-infinite-AFP}.

\tableofcontents

\section{Preliminaries}\label{preliminaries}

For any von Neumann algebra $M$, we will denote by $\mathcal Z(M)$ the centre of $M$, by $z_M(e)$ the central support of a projection $e \in M$, by $\mathcal U(M)$ the group of unitaries in $M$, by $\Ball(M)$ the unit ball of $M$ with respect to the uniform norm $\|\cdot\|_\infty$ and by $(M, \rL^2(M), J^M, \mathfrak P^M)$ the standard form of $M$. We will say that an inclusion of von Neumann algebras $P \subset 1_P M 1_P$ is {\em with expectation} if there exists a faithful normal conditional expectation $\rE_P : 1_P M 1_P \to P$. We will say that a $\sigma$-finite von Neumann algebra $M$ is {\em tracial} if it is endowed with a faithful normal tracial state $\tau$.

\subsection*{Background on $\sigma$-finite von Neumann algebras}

Let $M$ be any $\sigma$-finite von Neumann algebra with unique predual $M_\ast$ and $\varphi \in M_\ast$ any faithful state. We will write $\|x\|_\varphi = \varphi(x^* x)^{1/2}$ for every $x \in M$. Recall that on $\Ball(M)$, the topology given by $\|\cdot\|_\varphi$ 
coincides with the $\sigma$-strong topology. Denote by $\xi_\varphi \in \mathfrak P^M$ the unique representing vector of $\varphi$. The mapping $M \to \rL^2(M) : x \mapsto x \xi_\varphi$ defines an embedding with dense image such that $\|x\|_\varphi = \|x \xi_\varphi\|_{\rL^2(M)}$ for all $x \in M$.
 
We denote by $\sigma^\varphi$ the modular automorphism group of the state  $\varphi$.  The {\em centralizer} $M^\varphi$ of the state $\varphi$ is by definition the fixed point algebra of $(M, \sigma^\varphi)$.  The {\em continuous core} of $M$ with respect to $\varphi$, denoted by $\core_\varphi(M)$, is the crossed product von Neumann algebra $M \rtimes_{\sigma^\varphi} \R$.  The natural inclusion $\pi_\varphi: M \to \core_\varphi(M)$ and the unitary representation $\lambda_\varphi: \R \to \core_\varphi(M)$ satisfy the {\em covariance} relation
$$
  \lambda_\varphi(t) \pi_\varphi(x) \lambda_\varphi(t)^*
  =
  \pi_\varphi(\sigma^\varphi_t(x))
  \quad
  \text{ for all }
  x \in M \text{ and all } t \in \R.
$$
Put $\rL_\varphi (\R) = \lambda_\varphi(\R)\dpr$. There is a unique faithful normal conditional expectation $\rE_{\rL_\varphi (\R)}: \core_{\varphi}(M) \to \rL_\varphi(\R)$ satisfying $\rE_{\rL_\varphi (\R)}(\pi_\varphi(x) \lambda_\varphi(t)) = \varphi(x) \lambda_\varphi(t)$ for all $x \in M$ and all $t \in \R$. The faithful normal semifinite weight defined by $f \mapsto \int_{\R} \exp(-s)f(s) \, {\rm d}s$ on $\rL^\infty(\R)$ gives rise to a faithful normal semifinite weight $\Tr_\varphi$ on $\rL_\varphi(\R)$ {\em via} the Fourier transform. The formula $\Tr_\varphi = \Tr_\varphi \circ \rE_{\rL_\varphi (\R)}$ extends it to a faithful normal semifinite trace on $\core_\varphi(M)$.

Because of Connes's Radon--Nikodym cocycle theorem \cite[Th\'eor\`eme 1.2.1]{Co72} (see also \cite[Theorem VIII.3.3]{Ta03}), the semifinite von Neumann algebra $\core_\varphi(M)$ together with its trace $\Tr_\varphi$ does not depend on the choice of $\varphi$ in the following precise sense. If $\psi \in M_\ast$ is another faithful state, there is a canonical surjective $\ast$-isomorphism
$\Pi_{\varphi,\psi} : \core_\psi(M) \to \core_{\varphi}(M)$ such that $\Pi_{\varphi,\psi} \circ \pi_\psi = \pi_\varphi$ and $\Tr_\varphi \circ \Pi_{\varphi,\psi} = \Tr_\psi$. Note however that $\Pi_{\varphi,\psi}$ does not map the subalgebra $\rL_\psi(\R) \subset \core_\psi(M)$ onto the subalgebra $\rL_\varphi(\R) \subset \core_\varphi(M)$ (and hence we use the symbol $\rL_\varphi(\R)$ instead of the usual $\rL(\R)$).

We start with a rather technical lemma.

\begin{lem}\label{lemma-equivalence}
Let $M$ be any $\sigma$-finite von Neumann algebra endowed with any faithful state $\varphi \in M_\ast$. Then for any projection $p \in M$, there exists a projection $q \in M^\varphi$ such that $p \sim q$ in $M$. 
\end{lem}

\begin{proof}
Replacing $M$ with $M z_M(p)$ with the central support $z_M(p)$, we may and will assume that $z_M(p) = 1$. By \cite[Proposition 6.3.7]{KR97}, one can decompose $p = p_1 + p_2$ along $M = M_1\oplus M_2$ so that $p_1$ is finite and $p_2$ is properly infinite. Since $M$ is $\sigma$-finite,  $p_2$ is equivalent to $1_{M_2}$, which clearly belongs to $M^\varphi$. Hence we may and will assume that $p$ is finite with $z_M(p) = 1$ and hence $M$ is semifinite. Write $\varphi = \Tr(h\,\cdot\,)$ for some nonsingular positive selfadjoint operator $h$ affiliated with $M$ and take a MASA $A \subset M$ that contains $\{h^{{\rm i}t}\,|\,t \in \R\}\dpr$. We have $A \subset M^\varphi$. Since $A$ is a MASA with expectation, $A$ is generated by finite projections in $M$ (see e.g.\ \cite[Proposition 4.4]{To71} but this case can be proved without such a general assertion). We will prove that $p$ is equivalent in $M$ to a projection in $A$. Thanks to \cite[Corollaries 3.8, 3.13]{Ka82}, we may and will assume, by decomposing $M$ into the components of type ${\rm I}_n$, ${\rm II_1}$ and ${\rm II_\infty}$, that $M$ is of type ${\rm II_\infty}$. Here is a claim.

\begin{claim}
For any nonzero finite projection $e \in M$ and any nonzero projection $f \in A$ such that $z_M(e) z_M(f) \neq 0$, there exist nonzero projections $e' \in eMe$ and $f' \in Af$ such that $e'$ is equivalent to $f'$ in $M$. 
\end{claim}

\begin{proof}[Proof of the Claim]
As we observed before, there is an increasing sequence of projections $r_n \in A$ that are finite in $M$ and such that $r_n \to 1$ $\sigma$-strongly. By assumption, there exists $x \in M$ such that $exf \neq 0$. Then there exists $n_0$ so that $exfr_{n_0} \neq 0$. Taking the polar decomposition of the element $exfr_{n_0}$, we can find a nonzero subprojection $e'$ of $e$ such that $e'$ is equivalent in $M$ to a subprojection $s$ of $fr_{n_0}$. Observe that $s$ may not be in $A$. Hence we have to work further. Consider the MASA $Afr_{n_0}$ in the type ${\rm II_1}$ von Neumann subalgebra $fr_{n_0}Mfr_{n_0}$. By \cite[Proposition 3.13]{Ka82}, we can find a projection $f' \in Afr_{n_0} \subset A$ that is equivalent to $s$ in $M$.
\end{proof}

By Zorn's lemma, let $((p_i,q_i))_{i\in I}$ be a maximal family of pairs of projections such that $(p_i)_{i\in I}$ and $(q_i)_{i\in I}$ are families of pairwise orthogonal projections, all $p_i$ are subprojections of $p$, all $q_i$ are in $A$ and $p_i \sim q_i$ for every $i\in I$. Suppose that $e:= p - \sum_{i \in I} p_i \neq 0$ and put $f := 1 - \sum_{i\in I} q_i$. Observe that the central support of $f$ must be equal to $1$ since $\sum_{i\in I} q_i \sim \sum_{i\in I} p_i \leq p$ is finite and $M$ is of type ${\rm II_\infty}$  and hence properly infinite. Therefore, by the above claim, there exist nonzero projections $p_0, q_0$ such that $p_0 \leq e$, $q_0 \leq f$, $q_0 \in A$ and $p_0 \sim q_0$, a contradiction to the maximality of the family $((p_i,q_i))_{i\in I}$. Consequently, $p = \sum_{i \in I} p_i \sim \sum_{i \in I} q_i \in A$. Hence we are done. 
\end{proof}

The following simple application of the previous lemma will turn out to be useful for Popa's intertwining techniques in the type ${\rm III}$ setting.

\begin{prop}\label{prop-with-expectation} Let $A \subset M$ be any unital inclusion of $\sigma$-finite von Neumann algebras with expectation and $p \in A' \cap M$ any nonzero projection. Then $Ap \subset pMp$ is also with expectation. 
\end{prop} 
\begin{proof} By assumption we may choose a faithful state $\psi \in M_\ast$ such that  $A$ is globally invariant under the modular automorphism group $\sigma^\psi$  and, in particular, so is $A' \cap M$. Put $\varphi := \psi|_{A' \cap M}$ and observe that $(A' \cap M)^\varphi \subset M^\psi$. Applying Lemma \ref{lemma-equivalence} to $p \in A' \cap M$ with $\varphi$ we obtain a partial isometry $v \in A'\cap M$ such that $vv^* = p$ and $v^*v \in (A' \cap M)^\varphi \subset M^\psi$, the latter of which shows that $Av^*v \subset v^*v M v^*v$ is with expectation. Since $v \in A' \cap M$, the inclusions $Av^*v \subset v^*v M v^*v$ and $Ap \subset pMp$ are conjugate to each other via $\Ad(v)$ and hence $Ap \subset pMp$ is with expectation.   
\end{proof}

Recall that for any inclusion of von Neumann algebras $A \subset M$, the {\em group of normalizing unitaries} is defined by 
$$\mathcal N_M(A) = \{u \in \mathcal U(M) : uAu^* = A\}.$$
The von Neumann algebra $\mathcal N_M(A)\dpr$ is  called the {\em normalizer of} $A$ {\em inside} $M$. The next result is a variation on \cite[Lemma 3.5]{Po03} and will be used in the proof of  Theorem \ref{thm-normalizer-general-AFP}.

\begin{prop}\label{prop-corner-normalizer}
Let $M$ be any $\sigma$-finite von Neumann algebra and $A \subset M$ any von Neumann subalgebra. Assume moreover that $A' \cap M = \mathcal{Z}(A)$.  Then for any nonzero projection $p \in \mathcal{Z}(A)$, we have $$\mathcal{N}_{pMp}(Ap)\dpr = p\big(\mathcal{N}_M(A)\dpr\big)p.$$ 
\end{prop}

\begin{proof} 
For any $u \in \mathcal{N}_{pMp}(Ap)$, we have $v := u + (1-p) \in \mathcal{N}_M(A)$ and $pvp = u$. Thus, $\mathcal{N}_{pMp}(Ap) \subset p\mathcal{N}_M(A)p$ and hence the inclusion ($\subset$) holds without taking double commutant. Therefore, it suffices to prove the reverse inclusion relation. 

Write $N :=  \mathcal{N}_M(A)\dpr$ for simplicity. Let $u \in \mathcal{N}_M(A)$ be an arbitrary element. Set $v := pup$. Since $\Ad(u)|_{A}$ gives a unital $\ast$-automorphism of $A$, we have $u\mathcal{Z}(A)u^* = \mathcal{Z}(A)$ so that $v^* v = u^* p u \, p$ and $vv^* = u p u^* \, p$ are projections in $\mathcal{Z}(A)$. In particular, $v$ is a partial isometry. Moreover, it is plain to see that for each $a \in A$, we have 
$$
vav^* = (uau^*)(upu^*)p \in Av v^* \quad \text{and} \quad v^* av = (u^*au)(u^* p u)p \in Av^*v.
$$
Hence, $vAv^* = Avv^*$ and $v^* A v = Av^* v$. Observe that $\mathcal{Z}(N) \subset A' \cap M = \mathcal{Z}(A)$.

\begin{claim}
There exists a partial isometry $w \in N$ such that (i) $w^* w, ww^* \in \mathcal{Z}(A)p$, (ii) $wAw^* = Aww^*$, $w^* Aw = Aw^* w$, (iii) $wv^* v = v = vv^* w$, and moreover that (iv) with letting $z := z_N(w^*w)=z_N(ww^*) \in \mathcal Z(A)p$ (see the notation at the beginning of this section), there exist orthogonal projections $z_1, z_2, z_3 \in \mathcal{Z}(N)$ with $z_1 + z_2 +z_3 = z$ so that 
\begin{itemize} 
\item $w^* w z_1 = z_1$ but $ww^* z_1 \lneqq z_1$, 
\item $w^*w z_2 \lneqq z_2$ but $ww^* z_2 = z_2$ and
\item $w^*wz_3 = z_3 = ww^* z_3$.
\end{itemize} 
\end{claim}

\begin{proof}[Proof of the Claim]
To this end, choose a maximal family of partial isometries $w_i \in M$ such that $(w^*_i w_i)_i$ and $(w_i w_i^*)_i$ are families of pairwise orthogonal projections, $w_i A w_i^* = Aw_i w_i^*$, $w_i^* A w_i = Aw_i^* w_i$, $w_i^* w_i \leq p - v^*v$ and $w_i w_i^* \leq p - vv^*$. Then $w := v + \sum_i w_i$ clearly enjoys (i)--(iii). 

Choose a maximal orthogonal family of projections $z_{3i} \in \mathcal{Z}(N)z$ such that $w^*wz_{3i} = z_{3i} = ww^* z_{3i}$. Set $z_3 := \sum_i z_{3i}$. Choose a maximal orthogonal family of projections $z_{2j} \in \mathcal{Z}(N)(z-z_3)$ such that $ww^* z_{2j} = z_{2j}$. Set $z_2 := \sum_j z_{2j}$. By construction, we have $w^*wz_{3} = z_{3} = ww^* z_{3}$ and $w^*wz_2 \lneqq z_2 = ww^* z_2$. Set $z_1 := z-z_2-z_3$. Assume that $z_1 \neq 0$; otherwise we are already done. By the maximality of the families $(z_{2j})_j$ and $(z_{3i})_i$, observe that no nonzero projection $z' \in \mathcal{Z}(N)z_1$ enjoys $ww^* z' = z'$. This means that the central support of $z_1 - ww^*z_1$ in $N$ is equal to $z_1$. Suppose that $w^* w z_1 \lneqq z_1$. Then $(z_1 - ww^* z_1)N(z_1 - w^*w z_1) \neq \{0\}$ must hold. Hence there exists $x \in \mathcal{N}_M(A)$ such that $(z_1 - ww^* z_1)x(z_1 - w^*w z_1) \neq 0$. Observe that $z_1 \in \mathcal{Z}(N) \subset A' \cap M = \mathcal{Z}(A)$. Thus the first part (dealing with the $v$) shows that $w_0 := (z_1 - ww^* z_1)x(z_1 - w^*w z_1) \in N$ is a new nonzero partial isometry such that $w_0^* w_0 \in \mathcal{Z}(A)(z_1 - w^*w z_1)$, $w_0 w_0^* \in \mathcal{Z}(A)(z_1 - ww^* z_1)$, $w_0 Aw^*_0 = Aw_0 w^*_0$ and $w_0^* A w_0 = Aw_0^* w_0$, a contradiction due to the maximality of the family $(w_i)_i$. Hence $w^* w z_1 = z_1$ (and $ww^* z_1 \lneqq z_1$). Thus we have proved the claim. 
\end{proof}

Write $w_k := wz_k$, $k=1,2,3$. Observe that $\mathcal{Z}(N) \subset A' \cap M = \mathcal{Z}(A)$  and hence each $w_k$, in place of $w$, satisfies (i)--(ii) in the above claim. We will first deal with $w_1$ when it is nonzero. Set $e_1 := z_1-w_1w_1^* \neq 0$ and  $e_i:= w_1^{i-1} e_1 w_1^*{}^{i-1}$, $i=2,3,\dots$. Observe that all the projections $e_n$ are in $\mathcal{Z}(A)z_1$, since $\Ad(w_1)|_{Az_1}$ defines a unital $\ast$-isomorphism between $Az_1$ and $Aw_1w_1^*$ with $w_1w_1^* \in \mathcal{Z}(A)$. We claim that the projections $e_n$ are pairwise orthogonal. Indeed, if $i \lneqq j$, we have $0 \leq e_i e_j = w_1^{i-1}e_1 w_1^*{}^{i-1} w_1^{j-1}e_1 w_1^*{}^{j-1} = w_1^{i-1} (e_1 w_1^{j-i} e_1 w_1^*{}^{j-i})w_1^*{}^{i-1} \leq w_1^{i-1} ((z_1-w_1w_1^*)(w_1w_1^*))w_1^*{}^{i-1} = 0$ so that $e_i e_j = 0$. We also claim that $w_1 fw_1^* = f$ with $f := z_1-\sum_{n\geq1} e_n$. Indeed, $w_1 f w_1^* = w_1 w_1^* - \sum_{n\geq2}e_n = z_1 - (z_1- w_1w_1^*) - \sum_{n\geq2}e_n = z_1 - \sum_{n\geq1} e_n = f$. Put $w_1(n) := w_1(\sum_{i=1}^{n-1}e_i) + w_1^*{}^{n-1}e_n + \sum_{i\geq n+1}e_i + w_1 f + (p-z_1)$. Clearly, all the elements $w_1(n)$ are in $\mathcal{N}_{pMp}(Ap)$ and $w_1(n)z_1 = z_1 w_1(n)$ converges to $w_1$ as $n \to \infty$ and hence $w_1 \in \mathcal{N}_{pMp}(Ap)\dpr$. Similarly, we can prove that $w_2^* \in \mathcal{N}_{pMp}(Ap)\dpr$, implying that $w_2 \in \mathcal{N}_{pMp}(Ap)\dpr$. Finally, it is trivial that $w_3 + (p-z_3) \in \mathcal{N}_{pMp}(Ap)$, implying that $w_3 \in \mathcal{N}_{pMp}(Ap)\dpr$. Consequently, we have $v = vv^*w = vv^*(w_1 + w_2 + w_3) \in \mathcal{N}_{pMp}(Ap)\dpr$. Hence we are done.   
\end{proof}

We point out that we do not need to assume the inclusion $A \subset M$ to be with expectation in Proposition \ref{prop-corner-normalizer}.

\subsection*{Popa's intertwining techniques}

To fix notation, let $M$ be any $\sigma$-finite von Neumann algebra, $1_A$ and $1_B$ any nonzero projections in $M$, $A\subset 1_AM1_A$ and $B\subset 1_BM1_B$ any von Neumann subalgebras. Popa introduced his powerful {\em intertwining-by-bimodules techniques} in \cite{Po01} in the case when $M$ is finite and more generally in \cite{Po03} in the case when $M$ is endowed with an almost periodic faithful normal state $\varphi$ for which $1_A \in M^\varphi$, $A \subset 1_A M^\varphi 1_A$ and $1_B \in M^\varphi$, $B \subset 1_B M^\varphi 1_B$. It was showed in \cite{HV12,Ue12} that Popa's intertwining techniques extend to the case when $B$ is finite and with expectation in $1_B M 1_B$ and $A \subset 1_A M 1_A$ is any von Neumann subalgebra.

In this paper, we will need the following generalization of \cite[Theorem A.1]{Po01} in the case when $A \subset 1_A M 1_A$ is any finite von Neumann subalgebra with expectation and $B \subset 1_B M 1_B$ is any von Neumann subalgebra with expectation.

\begin{thm}[{\cite[Theorem 4.3]{HI15}}]\label{theorem-intertwining}
Let $M$ be any $\sigma$-finite von Neumann algebra, $1_A$ and $1_B$ any nonzero projections in $M$, $A\subset 1_AM1_A$ and $B\subset 1_BM1_B$ any von Neumann subalgebras with faithful normal conditional expectations $\rE_A : 1_A M 1_A \to A$ and $\rE_B : 1_B M 1_B \to B$ respectively. Assume moreover that $A$ is a finite von Neumann algebra.

Then the following conditions are equivalent:

\begin{enumerate}
\item There exist projections $e \in A$ and  $f \in B$, a nonzero partial isometry $v \in eMf$ and a unital normal $\ast$-homomorphism $\theta : eAe \to fBf$ such that the inclusion $\theta(eAe) \subset fBf$ is with expectation and $av = v \theta(a)$ for all $a \in eAe$.
\item There exist $n \geq 1$, a projection $q \in \mathbf M_n (B)$, a nonzero partial isometry $v\in \mathbf{M}_{1, n}(1_A M)q$ and a unital normal $\ast$-homomorphism $\pi\colon A \rightarrow q\mathbf{M}_n(B)q$ such that the inclusion $\pi(A) \subset q \mathbf M_n (B) q$ is with expectation and $av = v\pi(a)$ for all $a\in A$. 
\item There exists no net $(w_i)_{i \in I}$ of unitaries in $\mathcal U(A)$ such that $\lim_i \rE_{B}(b^*w_i a) = 0$ $\sigma$-strongly for all $a,b\in 1_AM1_B$.
\end{enumerate}
If one of the above conditions is satisfied, we will say that $A$ {\em embeds with expectation into} $B$ {\em inside} $M$ and write $A \preceq_M B$.
\end{thm}

Moreover, \cite[Theorem 4.3]{HI15} asserts that when $B \subset 1_B M 1_B$ is a {\em semifinite} von Neumann subalgebra endowed with any fixed faithful normal semifinite trace $\Tr$, then $A \preceq_M B$ if and only if there exist a projection $e \in A$, a $\Tr$-finite projection $f \in B$, a nonzero partial isometry $v \in eMf$ and a unital normal $\ast$-homomorphism $\theta : eAe \to fBf$ such that $av = v \theta(a)$ for all $a \in eAe$. Hence, in that case, the notation $A \preceq_M B$ is consistent with \cite[Proposition 3.1]{Ue12}. In particular, the projection $q \in \mathbf M_n (B)$ in Theorem \ref{theorem-intertwining} (2) is chosen to be finite under the trace $\Tr \otimes \tr_n$, when $B$ is semifinite with any fixed faithful normal semifinite trace $\Tr$. We refer to \cite[Section 4]{HI15} for further details.

\begin{rem}\label{remark-intertwining}
Keep the notation of Theorem \ref{theorem-intertwining}. 
\begin{enumerate}
\item Proposition \ref{prop-with-expectation} gives the following useful additional facts to Theorem \ref{theorem-intertwining}: The inclusions $eAe \, vv^* \subset vv^* M vv^*$ and $\theta(eAe) \, v^*v \subset v^*v M v^*v$ in (2) are also with expectation. Likewise, the inclusions $A  ww^* \subset ww^* M ww^*$ and $\pi(A)  w^*w \subset w^*w \mathbf M_n(M) w^*w$ in (3) are also with expectation.

\item Assume that there exist $k \geq 1$ and a nonzero partial isometry $u \in \mathbf M_{1, k}(M)$ such that $uu^* \in A' \cap 1_A M 1_A$ and $u^* A u \preceq_{\mathbf M_k(M)} \mathbf M_k(B)$. Then $A \preceq_M B$ holds. Indeed, there exist $n \geq 1$, a projection $q \in \mathbf M_{n}(\mathbf M_k(M))$, a nonzero partial isometry $w \in \mathbf M_{1, n}(u^* u \mathbf M_k(M))q$ and a unital normal $\ast$-homomorphism $\pi : u^* A u \to q \mathbf M_n(\mathbf M_k(B)) q$ such that the unital inclusion $\pi(u^* A u) \subset q\mathbf M_n(\mathbf M_k(B))q$ is with expectation and $y w = w \pi(y)$ for all $y \in u^* A u$. Define the unital normal $\ast$-homomorphism $\iota : A \to u^* Au : a \mapsto u^* a u$. Then a simple computation shows that $a \, uw = uw \, (\pi \circ \iota)(a)$ for all $a \in A$, where $uw \in \mathbf M_{1, nk}(1_AM)q$ and $uw \neq 0$, $\pi \circ \iota : A \to q\mathbf M_{nk}(B) q$ is a unital normal $\ast$-homomorphism and the unital inclusion $(\pi \circ \iota)(A) \subset q\mathbf M_{nk}(B) q$ is with expectation. Therefore, we obtain $A \preceq_M B$.
\end{enumerate}
\end{rem}

We are also going to use the following useful technical lemma. This is a generalization of \cite[Remark 3.8]{Va07}.

\begin{lem}\label{lemma-intertwining}
Keep the notation of Theorem \ref{theorem-intertwining}. Let $B \subset P \subset 1_P M 1_P$ be any intermediate von Neumann subalgebra with expectation. Assume that $A \preceq_M P$ and $A \npreceq_M B$.

Then there exist $k \geq 1$, a projection $q \in \mathbf M_k(P)$, a nonzero partial isometry $w \in \mathbf M_{1, k}(1_A M)q$ and a unital normal $\ast$-homomorphism $\pi : A \to q\mathbf M_k(P)q$ such that the unital inclusion $\pi(A) \subset q\mathbf M_k(P)q$ is with expectation, $\pi(A) \npreceq_{\mathbf M_k(P)} \mathbf M_k(B)$ and $a w = w \pi(a)$ for all $a \in A$.
\end{lem}

\begin{proof}
Since $A \preceq_M P$, there exist $k \geq 1$, a projection $q \in \mathbf M_k(P)$, a nonzero partial isometry $w \in \mathbf M_{1, k}(1_A M)q$ and a unital normal $\ast$-homomorphism $\pi : A \to q\mathbf M_k(P)q$ such that the unital inclusion $\pi(A) \subset q\mathbf M_k(P)q$ is with expectation and $a w = w \pi(a)$ for all $a \in A$. We have $w^*w \in \pi(A)' \cap q \mathbf M_k(M) q$. Following  \cite[Remark 3.8]{Va07}, denote by $q_0$ the support projection (belonging to $q\mathbf M_k(P)q$)  of the element $\rE_{q\mathbf M_k(P)q}(w^*w)$ and observe that $q_0 \in \pi(A)' \cap q\mathbf M_k(P)q$. Observe that $\rE_{q\mathbf M_k(P)q}((q-q_0)w^*w(q-q_0)) = 0$, and hence $w(q-q_0) = 0$, that is, $w = wq_0$. Thanks to Proposition \ref{prop-with-expectation}, replacing $q$ and $\pi$ with $q_0$ and $\pi(\cdot)q_0$, respectively, we may assume without loss of generality that $q$ is equal to the support projection of the element $\rE_{q\mathbf M_k(P)q}(w^*w)$.

We claim that we have $\pi(A) \npreceq_{\mathbf M_k(P)} \mathbf M_k(B)$. Indeed, otherwise there exist $n \geq 1$, a projection $r \in \mathbf M_{n}(\mathbf M_k(B))$, a nonzero partial isometry $u \in \mathbf M_{1, n}(q \mathbf M_k(P))r$ and a unital normal $\ast$-homomorphism $\theta : \pi(A) \to r\mathbf M_n(\mathbf M_k(B))r$ such that the unital inclusion $(\theta \circ \pi)(A) \subset r\mathbf M_n(\mathbf M_k(B))r$ is with expectation and $b u = u \theta(b)$ for all $b \in \pi(A)$. We moreover have $a \, w u = w u \, (\theta \circ \pi)(a)$ for all $a \in A$. Observe that $w u \neq 0$. Indeed, otherwise we have $w u = 0$ and hence 
$$
\rE_{\mathbf M_k(P)}(w^*w) u = \rE_{\mathbf M_n(\mathbf M_k(P))}(w^*w \, u) = 0.
$$
Since $q$ is equal to the support projection of the element $\rE_{q\mathbf M_k(P)q}(w^*w)$ and since $u \in \mathbf M_{1, n}(q \mathbf M_k(P))r$, this implies that $q u = 0$ and hence $u = 0$, a contradiction. Therefore, we have $w u \neq 0$ and hence $A \preceq_M B$, a contradiction. Consequently, we obtain $\pi(A) \npreceq_{\mathbf M_k(P)} \mathbf M_k(B)$.
\end{proof}

We point out that when $P \subset 1_P M 1_P$ is a {\em semifinite} von Neumann subalgebra endowed with a faithful normal semifinite trace $\Tr$, we may choose the nonzero projection $q \in \mathbf M_k(P)$ appearing in Lemma \ref{lemma-intertwining} to be of finite trace with respect to the faithful normal trace $\Tr \otimes \tr_k$.

\subsection*{Amalgamated free product von Neumann algebras}

Let $I$ be any nonempty set and $(B \subset M_i)_{i \in I}$ any family of inclusions of $\sigma$-finite von Neumann algebras with faithful normal conditional expectations $\rE_i : M_i \to B$. The {\em amalgamated free product} $(M, \rE) = \ast_{B, i \in I} (M_i, \rE_i)$ is the unique pair of von Neumann algebra $M$ generated by $(M_i)_{i \in I}$ and faithful normal conditional expectation $\rE : M \to B$ such that $(M_i)_{i \in I}$ is {\em freely independent} with respect to $\rE$:  
$$\rE(x_1 \cdots x_n) = 0 \; \text{ whenever } \; x_j \in M_{i_j}^\circ, \; i_1, \dots, i_n \in I \; \text{ and } \; i_1 \neq \cdots \neq  i_{n}.$$
Here and in what follows, we denote by $M_i^\circ := \ker(\rE_i)$. We call the resulting $M$ the {\em amalgamated free product von Neumann algebra} of $(M_i,E_i)_{i\in I}$ over $B$. We refer to the product $x_1 \cdots x_n$ where $x_j \in M_{i_j}^\circ$, $i_1, \dots, i_n \in I$ and $i_1 \neq \cdots \neq i_{n}$ as a {\em reduced word} in $M_{i_1}^\circ \cdots M_{i_n}^\circ$ of {\em length} $n \geq 1$. The linear span of $B$ and of all the reduced words in $M_{i_1}^\circ \cdots M_{i_n}^\circ$ where $n \geq 1$, $i_1, \dots, i_n \in I$ and $i_1 \neq \cdots \neq i_{n}$ forms a unital $\sigma$-strongly dense $\ast$-subalgebra of $M$.

When $B = \C 1$, $\rE_i = \varphi_i(\cdot) 1$ for all $i \in I$ and $\rE = \varphi(\cdot) 1$, we will simply  write $(M, \varphi) = \ast_{i \in I} (M_i, \varphi_i)$ and call the resulting $M$ the {\em free product von Neumann algebra} of $(M_i,\varphi_i)_{i\in I}$. 

When $B$ is a semifinite von Neumann algebra with faithful normal semifinite trace $\Tr$ and the weight $\Tr \circ \rE_i$ is tracial on $M_i$ for every $i \in I$, the weight $\Tr \circ \rE$ is  tracial on $M$ (see \cite[Proposition 3.1]{Po90} for the finite case and \cite[Theorem 2.6]{Ue98a} for the general case). In particular, $M$ is a semifinite von Neumann algebra. In that case, we will refer to $(M, \rE) = \ast_{B, i \in I} (M_i, \rE_i)$ as a {\em semifinite} amalgamated free product.

Let $\varphi \in B_\ast$ be any faithful state. Then for all $t \in \R$, we have $\sigma_t^{\varphi \circ \rE} = \ast_{i \in I} \sigma_t^{\varphi \circ \rE_i}$ (see \cite[Theorem 2.6]{Ue98a}). By \cite[Theorem IX.4.2]{Ta03}, for every $i \in I$, there exists a unique $\varphi\circ \rE$-preserving conditional expectation $\rE_{M_i} : M \to M_i$. Moreover, we have $\rE_{M_i}(x_1 \cdots x_n) = 0$ for all the reduced words $x_1 \cdots x_n$ that contain at least one letter from $M_j^\circ$ for some $j \in I \setminus \{i\}$ (see e.g.\ \cite[Lemma 2.1]{Ue10}). We will denote by $M \ominus M_i := \ker (\rE_{M_i})$. For more on (amalgamated) free product von Neumann algebras, we refer the reader to \cite{BHR12, Po90, Ue98a, Ue10, Ue12, Vo85, VDN92}.

The next lemma is a variant of \cite[Lemma 2.6]{HU15}. 

\begin{lem}\label{lemma-control-sequence}
For each $i \in \{ 1, 2 \}$, let $B \subset M_i$ be any inclusion of $\sigma$-finite von Neumann algebras with faithful normal conditional expectation $\rE_i : M_i \to B$. Denote by $(M, \rE) = (M_1, \rE_1) \ast_B (M_2, \rE_2)$ the corresponding amalgamated free product. 

Let $\psi \in M_\ast$ be any faithful state such that $\psi = \psi \circ \rE_{M_1}$. Let $(u_j)_{j \in J}$ be any net in $\Ball((M_1)^\psi)$ such that $\lim_j \rE_1(b^* u_j a) = 0$ $\sigma$-strongly for all $a, b \in M_1$. Then for all $x, y \in M$, we have $\lim_j \rE_{M_2}(y^* u_j x) = 0$ $\sigma$-strongly.
\end{lem}

\begin{proof}
We first prove the $\sigma$-strong convergence when $x, y \in M_1 \cup M_1 M_2^\circ \cdots M_2^\circ M_1$ are {\em words} of the form $x = a x' c$ or $x = a$ and $y = b y' d$ or $y = b$ with $a, b, c, d \in M_1$ and $x', y' \in M_2^\circ \cdots M_2^\circ$. By free independence, for all $j \in J$, we have
$$\rE_{M_2}(y^* u_j x) = 
\begin{cases} 
\rE_{M_2}(d^*y'^* \, \rE_1(b^* u_j a)  \,  x'c) & (x = a x' c, y = b y' d), \\ 
\rE_{M_2}(d^*y'^*)\rE_1(b^* u_j a) & (x = a, y = b y' d), \\
\rE_1(b^* u_j a)\rE_{M_2}(x'c) & (x = a x' c, y = b), \\
\rE_1(b^* u_j a) & (x = a, y = b).
\end{cases}$$
Since $\lim_j \rE_1(b^* u_j a) = 0$ $\sigma$-strongly, we have $\lim_j \rE_{M_2}(y^* u_j x) = 0$ $\sigma$-strongly. 

We combine now the same pattern of approximation as in the proof of \cite[Lemma 2.6]{HU15} with a trick using standard forms as in the proof of \cite[Theorem 3.1]{HU15}. Namely, we will work with the standard form $(M, \rL^2(M), J^M, \mathfrak P^M)$ and denote by $e_{M_2}$ the Jones projection determined by $\rE_{M_2}$. Choose a faithful state $\varphi \in M_\ast$ with $\varphi = \varphi\circ \rE$. Denote by $\xi_\varphi, \xi_\psi \in \mathfrak P^M$ the unique representing vectors of $\varphi, \psi$ respectively. Observe that $\varphi = \varphi \circ \rE_{M_2}$ and hence $e_{M_2}x\xi_\varphi = \rE_{M_2}(x)\xi_\varphi$ holds for every $x \in M$ (though we do {\em not} have $e_{M_2}x\xi_\psi = \rE_{M_2}(x)\xi_\psi$). The rest of the proof is divided into three steps.    

(First step) We first prove that $\lim_j \| e_{M_2}y^* u_j \xi\|_{\rL^2(M)} = 0$ for any $\xi \in \rL^2(M)$ and any {\em word} $y \in M_1 \cup M_1 M_2^\circ \cdots M_2^\circ M_1$. Indeed, we may choose a sequence $(x_k)_k$, where each $x_k$ is a finite linear combination of {\em words} in $M_1 \cup M_1 M_2^\circ \cdots M_2^\circ M_1$, and such that $\|\xi-x_k\xi_\varphi \|_{\rL^2(M)} \rightarrow 0$ as $k\to\infty$, since those linear combinations of words form a $\sigma$-strongly dense $*$-subalgebra of $M$. Then, for all $j \in J$ and $k \in \N$, we have
\begin{align*}
\|e_{M_2}y^* u_j \xi\|_{\rL^2(M)} 
&\leq \|e_{M_2}y^* u_j x_k\xi_\varphi\|_{\rL^2(M)} + \|e_{M_2}y^* u_j (\xi - x_k\xi_\varphi)\|_{\rL^2(M)} \\
&\leq \|\rE_{M_2}(y^* u_j x_k)\xi_\varphi\|_{\rL^2(M)} + \|y\|_\infty \|\xi - x_k\xi_\varphi\|_{\rL^2(M)}.
\end{align*}
The first part of the proof implies that $\limsup_j \|e_{M_2}y^* u_j \xi\|_{\rL^2(M)}  \leq \|y\|_\infty \|\xi - x_k\xi_\varphi\|_{\rL^2(M)}$ for all $k \in \N$ and hence $\lim_j \|e_{M_2}y^* u_j \xi\|_{\rL^2(M)} = 0$.

(Second step) We next prove that $\lim_j \|e_{M_2}y^* u_j x\xi_\psi\|_{\rL^2(M)} = 0$ for any \emph{analytic element} $x \in M$ with respect to the modular automorphism group $\sigma^\psi$ and any element $y \in M$. Indeed, we may choose a sequence $(y_k)_k$, where each $y_k$ is a finite linear combination of {\em words} in $M_1 \cup M_1 M_2^\circ \cdots M_2^\circ M_1$, and such that $\lim_{k\to\infty}\|y^*\xi_\psi - y_k^*\xi_\psi\|_{\rL^2(M)} = 0$. Then, for all $j \in J$ and  $k \in \N$, we have
\begin{align*}
\|e_{M_2}y^* u_j x\xi_\psi\|_{\rL^2(M)} &\leq \|e_{M_2}y^*_k u_j x\xi_\psi\|_{\rL^2(M)} + \|e_{M_2}(y^*-y^*_k) u_j x\xi_\psi\|_{\rL^2(M)} \\
&\leq \|e_{M_2}y^*_k u_j x\xi_\psi\|_{\rL^2(M)} + \|(y^* - y_k^*) u_j x\xi_\psi \|_{\rL^2(M)} \\
&= \|e_{M_2}y^*_k u_j x\xi_\psi\|_{\rL^2(M)} +  \|J^M \sigma_{{\rm i}/2}^\psi(x)^* u_j^* J^M(y^*\xi_\psi - y_k^*\xi_\psi) \|_{\rL^2(M)} \\
&\leq \|e_{M_2}y^*_k u_j x\xi_\psi\|_{\rL^2(M)} +  \|\sigma_{{\rm i}/2}^\psi(x)\|_\infty \|y^*\xi_\psi - y_k^*\xi_\psi \|_{\rL^2(M)},
\end{align*}
since $u_j \in (M_1)^\psi$. The first step implies that $\limsup_j \|e_{M_2}y^* u_j x\xi_\psi\|_{\rL^2(M)}  \leq \|\sigma_{{\rm i}/2}^\psi(x)\|_\infty \|y^*\xi_\psi - y_k^*\xi_\psi \|_{\rL^2(M)}$ for all $k \in \N$ and hence $\lim_j \|e_{M_2}y^* u_j x\xi_\psi\|_{\rL^2(M)} = 0$.

(Final step) We finally prove that $\lim_j \|\rE_{M_2}(y^* u_j x)\xi_\varphi\|_{\rL^2(M)} = 0$ for any elements $x, y \in M$. Indeed, we may choose a sequence $(x_k)_k$ in $M$ of analytic elements with respect to the modular automorphism group $\sigma^\psi$ such that $\lim_{k\to\infty}\|x\xi_\varphi - x_k\xi_\psi\|_{\rL^2(M)} = 0$. Then, for all $j\in J$ and $k \in \N$, we have
\begin{align*}
\|\rE_{M_2}(y^* u_j x)\xi_\varphi\|_{\rL^2(M)} 
&= \|e_{M_2}y^* u_j x\xi_\varphi\|_{\rL^2(M)} \\
&\leq \|e_{M_2}y^* u_j x_k\xi_\psi\|_{\rL^2(M)} + \|e_{M_2}y^* u_j (x\xi_\varphi - x_k\xi_\psi)\|_{\rL^2(M)} \\
&\leq \|e_{M_2}y^* u_j x_k\xi_\psi\|_{\rL^2(M)} + \|y\|_\infty \|x\xi_\varphi - x_k\xi_\psi\|_{\rL^2(M)}.
\end{align*}
The second step implies that $\limsup_j \|\rE_{M_2}(y^* u_j x)\xi_\varphi\|_{\rL^2(M)} \leq \|y\|_\infty \|x\xi_\varphi - x_k\xi_\psi\|_{\rL^2(M)}$ for all $k \in \N$ and hence $\lim_j \|\rE_{M_2}(y^* u_j x)\xi_\varphi\|_{\rL^2(M)} = 0$. Hence we are done.
\end{proof}

The next lemma will be used in the proof of the Main Theorem. This can be regarded as a variant of \cite[Corollary 4.3]{Po83}, \cite[Lemma 5.1]{Ge95} (in the tracial case), \cite[Proposition 6]{Ue98b} (in the non-tracial case) and also part of \cite[Theorem 1.1]{IPP05} (in the tracial amalgamated free product case). 

\begin{lem}\label{lemma-control} 
For each $i \in \{1, 2\}$, let $(M_i, \varphi_i)$ be any $\sigma$-finite von Neumann algebra endowed with any faithful normal state. Denote by $(M,\varphi) = (M_1,\varphi_1)\ast (M_2,\varphi_2)$ the corresponding free product.  

Let $1_Q \in M$ be any nonzero projection and $Q \subset 1_QM_11_Q$ be any diffuse von Neumann subalgebra with expectation. Let $n \geq 1$. If a partial isometry $v \in \mathbf{M}_{1,n}(M)$ with $vv^* \in Q$ or $vv^* \in Q' \cap 1_QM1_Q$ satisfies that $v^* Q v \subset \mathbf M_n(M_2)$, then $1_Q \, v = 0$. In particular, when $vv^* \in Q$, we have $v = 0$.  
\end{lem}

\begin{proof}
When $vv^* \in Q$, replacing $Q$ with $vv^* Q vv^*$ we may and will assume that $vv^* = 1_Q$. Hence, since $vv^* = 1_Q \in Q$ or $vv^* \in Q' \cap 1_Q M 1_Q$, we may think of the map $Q \to \mathbf M_n(M_2) : x \mapsto v^* x v$ as a normal (non-unital) $\ast$-homomorphism.   

Since $Q \subset 1_Q M_1 1_Q$ is with expectation, we may choose a faithful state $\psi \in M_\ast$ such that $\psi = \psi \circ \rE_{M_1}$, $1_Q \in (M_1)^\psi$, $Q \subset 1_Q M 1_Q$ is globally invariant under the modular automorphism group $\sigma^{\psi_Q}$ and $Q^{\psi_{Q}} \subset 1_Q (M_1)^\psi 1_Q$ is diffuse where $\psi_Q := \frac{\psi(1_Q \, \cdot \, 1_Q)}{\psi(1_Q)}$. See e.g.\ the proof of \cite[Lemma 2.1]{HU15}. 

Write $v = [v_1 \cdots v_n] \in \mathbf M_{1, n}(M)$ and denote by $\tr_n$ the canonical normalized trace on $\mathbf M_n(\C)$. Since $Q^{\psi_{Q}}$ is diffuse, we can choose a sequence of unitaries $(u_k)_k$ in $\mathcal U(Q^{\psi_Q})$ with $\lim_{k\to\infty} u_k = 0$ $\sigma$-weakly. By Lemma \ref{lemma-control-sequence}, we have 
$$\lim_{k\to\infty} \|\rE_{\mathbf M_n(M_2)}(v^* u_k v)\|_{\varphi \otimes \tr_n}^2 = \lim_{k\to\infty} \sum_{i, j = 1}^n \|\rE_{M_2}(v^*_i u_k v_j)\|_\varphi^2 = 0.$$
Since $v^* u_k v \in \mathcal U(v^* Q v) \subset \mathbf M_n(M_2)$, we have
$$\|v^* 1_Q v\|_{\varphi \otimes \tr_n} = \| v^* u_k v \, v^* 1_Q v \|_{\varphi \otimes \tr_n} 
= \| v^* u_k v \|_{\varphi \otimes \tr_n} = \|\rE_{\mathbf M_n(M_2)}(v^* u_k v)\|_{\varphi \otimes \tr_n} \rightarrow 0$$ as $k\to\infty$,
implying that $1_Q \, v = 0$.
\end{proof}

We point out that the above way of proof is applicable even to amalgamated free products over non-trivial subalgebras under suitable assumptions. Similarly, the same can be said about \cite[Proposition 2.7]{HU15}.

\subsection*{Ultraproduct von Neumann algebras}

Let $M$ be any $\sigma$-finite von Neumann algebra and $\omega \in \beta(\N) \setminus \N$ any nonprincipal ultrafilter. Define
\begin{align*}
\mathcal I_\omega(M) &= \left\{ (x_n)_n \in \ell^\infty(\N, M) : x_n \to 0\ \ast\text{-strongly as } n \to \omega \right\} \\
\mathcal M^\omega(M) &= \left \{ (x_n)_n \in \ell^\infty(\N, M) :  (x_n)_n \, \mathcal I_\omega(M) \subset \mathcal I_\omega(M) \text{ and } \mathcal I_\omega(M) \, (x_n)_n \subset \mathcal I_\omega(M)\right\}.
\end{align*}
The {\em multiplier algebra} $\mathcal M^\omega(M)$ is a C$^*$-algebra and $\mathcal I_\omega(M) \subset \mathcal M^\omega(M)$ is a norm closed two-sided ideal. Following \cite[\S 5.1]{Oc85}, we define the {\em ultraproduct von Neumann algebra} $M^\omega$ by $M^\omega := \mathcal M^\omega(M) / \mathcal I_\omega(M)$, which is indeed known to be a von Neumann algebra. We denote the image of $(x_n)_n \in \mathcal M^\omega(M)$ by $(x_n)^\omega \in M^\omega$. 

For every $x \in M$, the constant sequence $(x)_n$ lies in the multiplier algebra $\mathcal M^\omega(M)$. We will then identify $M$ with $(M + \mathcal I_\omega(M))/ \mathcal I_\omega(M)$ and regard $M \subset M^\omega$ as a von Neumann subalgebra. The map $\rE_\omega : M^\omega \to M : (x_n)^\omega \mapsto \sigma \text{-weak} \lim_{n \to \omega} x_n$ is a faithful normal conditional expectation. For every faithful state $\varphi \in M_\ast$, the formula $\varphi^\omega := \varphi \circ \rE_\omega$ defines a faithful normal state on $M^\omega$. Observe that $\varphi^\omega((x_n)^\omega) = \lim_{n \to \omega} \varphi(x_n)$ for all $(x_n)^\omega \in M^\omega$. 

Following \cite[\S2]{Co74}, we define
$$\mathcal M_\omega(M) := \left\{ (x_n)_n \in \ell^\infty(\N, M) : \lim_{n \to \omega} \|x_n \varphi - \varphi x_n\| = 0, \forall \varphi \in M_\ast \right\}.$$
We have $\mathcal I_\omega (M) \subset \mathcal M_\omega(M) \subset \mathcal M^\omega(M)$. The {\em asymptotic centralizer} $M_\omega$ is defined by $M_\omega := \mathcal M_\omega(M)/\mathcal I_\omega(M)$. We have $M_\omega \subset M^\omega$. Moreover, by \cite[Proposition 2.8]{Co74} (see also \cite[Proposition 4.35]{AH12}), we have $M_\omega = M' \cap (M^\omega)^{\varphi^\omega}$ for every faithful state $\varphi \in M_\ast$.

Let $Q \subset M$ be any von Neumann subalgebra with faithful normal conditional expectation $\rE_Q : M \to Q$. Choose a faithful state $\varphi \in M_\ast$ in such a way that $\varphi = \varphi \circ \rE_Q$. We have $\ell^\infty(\N, Q) \subset \ell^\infty(\N, M)$, $\mathcal I_\omega(Q) \subset \mathcal I_\omega(M)$ and $\mathcal M^\omega(Q) \subset \mathcal M^\omega(M)$. We will then identify $Q^\omega = \mathcal M^\omega(Q) / \mathcal I_\omega(Q)$ with $(\mathcal M^\omega(Q) + \mathcal I_\omega(M)) / \mathcal I_\omega(M)$ and be able to regard $Q^\omega \subset M^\omega$ as a von Neumann subalgebra. Observe that the norm $\|\cdot\|_{(\varphi |_Q)^\omega}$ on $Q^\omega$ is the restriction of the norm $\|\cdot\|_{\varphi^\omega}$ to $Q^\omega$. Observe moreover that $(\rE_Q(x_n))_n \in \mathcal I_\omega(Q)$ for all $(x_n)_n \in \mathcal I_\omega(M)$ and $(\rE_Q(x_n))_n \in \mathcal M^\omega(Q)$ for all $(x_n)_n \in \mathcal M^\omega(M)$. Therefore, the mapping $\rE_{Q^\omega} : M^\omega \to Q^\omega : (x_n)^\omega \mapsto (\rE_Q(x_n))^\omega$ is a well-defined conditional expectation satisfying $\varphi^\omega \circ \rE_{Q^\omega} = \varphi^\omega$. Hence, $\rE_{Q^\omega} : M^\omega \to Q^\omega$ is a faithful normal conditional expectation. For more on ultraproduct von Neumann algebras, we refer the reader to \cite{AH12, Oc85}.

We give a useful result showing how Popa's intertwining techniques behave with respect to taking ultraproduct von Neumann algebras.

\begin{prop}\label{proposition-ultraproducts}
Let $M$ be any $\sigma$-finite von Neumann algebra, $1_A$ and $1_B$ any nonzero projections in $M$, $A\subset 1_AM1_A$ and $B\subset 1_BM1_B$ any von Neumann subalgebras with faithful normal conditional expectations $\rE_A : 1_A M 1_A \to A$ and $\rE_B : 1_B M 1_B \to B$ respectively. Assume moreover that $A$ is a finite von Neumann algebra. 

Let $\omega \in \beta(\N) \setminus \N$ be any nonprincipal ultrafilter. Define $A^\omega \subset (1_A M 1_A)^\omega = 1_A M^\omega 1_A$ and $B^\omega \subset (1_B M 1_B)^\omega = 1_B M^\omega 1_B$. If $A^\omega \preceq_{M^\omega} B^\omega$, then $A \preceq_M B$.
\end{prop}

\begin{proof}
The proof uses an idea of \cite[Lemma 9.5]{Io12}. Choose a faithful state $\varphi \in M_\ast$ in such a way that $1_B \in M^\varphi$ and $\varphi_B \circ \rE_B = \varphi_B$ with $\varphi_B := \frac{\varphi(1_B\,\cdot\,1_B)}{\varphi(1_B)}$. Assume that $A^\omega \preceq_{M^\omega} B^\omega$. By Theorem \ref{theorem-intertwining}, there exist $\delta > 0$ and a finite subset $\mathcal F \subset  1_AM^\omega1_B$ such that 
\begin{equation}\label{equation-ultraproduct1}
\sum_{a, b \in \mathcal F} \|\rE_{B^\omega}(b^* u a)\|_{\varphi^\omega}^2 > \delta, \forall u \in \mathcal U(A^\omega).
\end{equation}
For each $a \in \mathcal F$, write $a = (a_n)^\omega$ with a fixed sequence $(a_n)_n \in 1_A\mathcal M^\omega( M)1_B$.

We next claim that there exists $n \in \N$ such that 
\begin{equation}\label{equation-ultraproduct2}
\sum_{a, b \in \mathcal F} \|\rE_{B^\omega}(b_n^* u a_n)\|_{\varphi^\omega}^2 \geq \delta, \forall u \in \mathcal U(A^\omega).
\end{equation}
Assume by contradiction that this is not the case. Then for every $n \in \N$, there exists $u_n \in \mathcal U(A^\omega)$ such that 
$$\sum_{a, b \in \mathcal F} \|\rE_{B^\omega}(b_n^* u_n a_n)\|_{\varphi^\omega}^2 < \delta.$$ 
Since $A$ is a finite von Neumann algebra, we may write $u_n = (u^{(n)}_m)^\omega$ with a sequence $(u^{(n)}_m)_m \in \ell^\infty(\N, A)$ such that $u_m^n \in \mathcal U(A)$ for all $m \in \N$. Then we have
$$\lim_{m \to \omega} \sum_{a, b \in \mathcal F} \|\rE_{B}(b_n^* u^{(n)}_m a_n)\|_\varphi^2 < \delta$$ for all $n \in \N$. 
Thus, we may choose $m_n \in \N$ large enough so that $v_n := u^{(n)}_{m_n}\in \mathcal U(A)$ satisfies 
$$\sum_{a, b \in \mathcal F} \|\rE_{B}(b_n^* v_n a_n)\|_\varphi^2 < \delta.$$
Since $A$ is finite, we may define $v := (v_n)^\omega \in \mathcal U(A^\omega)$ and we obtain
\begin{equation}\label{equation-ultraproduct3}
\sum_{a, b \in \mathcal F} \|\rE_{B^\omega}(b^* v a)\|_{\varphi^\omega}^2 = \lim_{n \to \omega} \sum_{a, b \in \mathcal F} \|\rE_{B}(b_n^* v_n a_n)\|_\varphi^2 \leq \delta.
\end{equation}
Equations \eqref{equation-ultraproduct1} and \eqref{equation-ultraproduct3} give a contradiction. This shows that Equation \eqref{equation-ultraproduct2} holds. Therefore, up to replacing the finite subset $\mathcal F \subset 1_A M^\omega 1_B$ with $\{a_n : a \in \mathcal F\} \subset 1_A M 1_B$, we may assume that $\mathcal F \subset 1_A M 1_B$ in Equation \eqref{equation-ultraproduct1}. In particular, we obtain
\begin{equation*}
\sum_{a, b \in \mathcal F} \|\rE_{B}(b^* u a)\|_{\varphi}^2  \geq  \delta, \forall u \in \mathcal U(A).
\end{equation*}
This finally implies that $A \preceq_M B$.
\end{proof}

\section{A characterization of von Neumann algebras with property Gamma}\label{characterization}

In this section, we generalize Popa's characterization of property Gamma for {\em tracial} von Neumann algebras (see \cite[Proposition 7]{Oz03} with $\mathcal N_0 = \mathcal M$) to {\em arbitrary} von Neumann algebras. This generalization is an unpublished result due to Houdayer--Raum, which they obtained through their recent work \cite{HR14}.

\begin{thm}\label{theorem-structure-gamma}
Let $M$ be any diffuse von Neumann algebra with separable predual and $\omega \in \beta(\N) \setminus \N$ any nonprincipal ultrafilter. The following conditions are equivalent.
\begin{itemize}
\item[(i)] The central sequence algebra $M' \cap M^\omega$ is diffuse. 
\item[(ii)] The asymptotic centralizer $M_\omega$ is diffuse.
\item[(iii)] There exists a faithful state $\psi \in M_\ast$ such that $M' \cap (M^\psi)^\omega$ is diffuse.
\item [(iv)] There exists a decreasing sequence $(A_n)_n$ of diffuse abelian von Neumann subalgebras of $M$ with expectation such that $M = \bigvee_{n \in \N} ((A_n)' \cap M)$.
\end{itemize}
\end{thm}

\begin{proof}
Let $z_k \in \mathcal Z(M)$ be a sequence of central projections such that $\sum_k z_k = 1$, $Mz_0$ has a diffuse center and $Mz_k$ is a diffuse factor for all $k \geq 1$. The equivalences (i) $\Leftrightarrow$ (ii) $\Leftrightarrow$ (iii) $\Leftrightarrow$ (iv) are all obvious for $Mz_0$, since all conditions actually hold true. Indeed, in order to obtain (iv), observe that it suffices to take $A_n = \mathcal Z(Mz_0)$ for every $n \in \N$. It remains to prove the equivalences for each $M z_k$ with $k \geq 1$. Therefore, in order to prove the result and without loss of generality, we may assume that $M$ is a diffuse factor.

(i) $\Rightarrow$ (ii) (c.f.\ \cite[Corollary 2.6]{HR14}.) Fix a faithful state $\varphi \in M_\ast$. By \cite[Proposition 2.8]{Co74} (see also \cite[Proposition 4.35]{AH12}), we have $M_\omega = M' \cap (M^\omega)^{\varphi^\omega}$. Then $M_\omega$ is diffuse by \cite[Theorem 2.3]{HR14} (see also \cite[Corollary 3.8]{Co74}).

(ii) $\Rightarrow$ (iii) Fix a faithful state $\varphi \in M_\ast$. Since $M' \cap (M^\omega)^{\varphi^\omega} = M_\omega$ is diffuse, we may choose a projection $e \in M_\omega$ such that $\varphi^\omega(e) = 2^{-1}$. Since $M$ is diffuse, we may write $e = (e_n)^\omega$ with a sequence of projections $(e_n)_n \in \mathcal M^\omega(M)$ such that $\varphi(e_n) = 2^{-1}$ for all $n \in \N$  (see \cite[Proposition 2.2]{HR14}). Observe that $\sigma\text{-weak} \lim_{n \to \omega} e_n = 2^{-1} 1_M$, since $M$ is a factor. Fix a countable $\|\cdot\|_\varphi$-dense subset $Y = \{y_n : n \in \N\} \subset M$.

Since $e \in M_\omega = M' \cap (M^\omega)^{\varphi^\omega}$, there exists $n \in \N$ large enough so that the projection $p_0 := e_n \in M$ satisfies $\varphi(p_0) = 2^{-1}$, $\|y_0 p_0 - p_0 y_0\|_\varphi \leq 2^{-1}$ and $\|\varphi  p_0 - p_0 \varphi\| \leq 2^{-1}$. Next, $ep_0 \in (M' \cap M^\omega)p_0 \subset (p_0 M p_0)^\omega$ is a projection satisfying $\varphi^\omega(ep_0) = \lim_{n \to \omega} \varphi(e_n p_0) = 2^{-2}$ because of $\sigma\text{-weak} \lim_{n \to \omega} e_n = 2^{-1} 1_M$. Since $p_0 M p_0$ is diffuse, we may write $ep_0 = (r_n)^\omega$ with a sequence of projections $(r_n)_n \in \mathcal M^\omega(p_0Mp_0)$ such that $\varphi(r_n) = 2^{-2}$ for all $n \in \N$. Likewise, we may write $ep_0^\perp = (s_n)^\omega$ for a sequence of projections $(s_n)_n \in \mathcal M^\omega(p_0^\perp M p_0^\perp)$ such that $\varphi(s_n) = 2^{-2}$ for all $n \in \N$. Observe that $e = ep_0 + ep_0^\perp = (r_n + s_n)^\omega \in M' \cap (M^\omega)^{\varphi^\omega}$. Then there exists $n \in \N$ large enough so that $p_1 := r_n + s_n$ satisfies $\varphi(p_1) = 2^{-1}$, $p_0 p_1 = p_1 p_0$, $\varphi(p_0 p_1) = 2^{-2}$, $\|y_j p_1 - p_1 y_j\| \leq 2^{-2}$ for all $0 \leq j \leq 1$ and $\| \varphi p_1 - p_1 \varphi\| \leq 2^{-2}$.

Repeating the above procedure, we construct by induction a sequence of projections $(p_n)_n$ in $M$ satisfying the following properties:
\begin{itemize}
\item[(P1)] $\varphi(p_n) = 2^{-1}$ for all $n \in \N$.
\item [(P2)] $p_j p_n = p_n p_j$ for all $j, n \in \N$.
\item [(P3)] $\varphi(p_{i_1} \cdots p_{i_r}) = 2^{-r}$ for all $r \geq 1$ and all $r$-tuples $(i_1, \dots, i_r)$ of pairwise distinct integers.
\item [(P4)] $\|y_j p_n - p_n y_j\|_\varphi \leq 2^{-(n + 1)}$ for all $0 \leq j \leq n$.
\item [(P5)] $\|\varphi p_n - p_n \varphi\| \leq 2^{-(n + 1)}$ for all $n \in \N$.
\end{itemize}
It follows that $(p_n)_n \in \mathcal M^\omega(M)$ and $p := (p_n)^\omega \in M' \cap (M^\omega)^{\varphi^\omega}$ satisfies $\varphi^\omega(p) = 2^{-1}$.

For each pair $0 \leq m \leq n$, put $\varphi_{m,n} := \sum_{j_m, \dots, j_n \in \{1, \perp\}} p_{m}^{j_m} \cdots p_{n}^{j_n} \, \varphi \, p_{m}^{j_m} \cdots p_{n}^{j_n} \in M_\ast$ and observe that $\varphi_{mn}$ is a faithful normal state. For any pair $0 \leq m \leq n$, using the triangle inequality with (P2) and (P5), we have
\begin{align*}
\|\varphi_{m,n} - \varphi_{m,n + 1}\| \leq  \left\|\varphi -  \sum_{j_{n + 1} \in \{1, \perp\}} p_{n + 1}^{j_{n + 1}} \, \varphi \, p_{n + 1}^{j_{n + 1}}\right\| 
 \leq 2 \, \|\varphi p_{n + 1} - p_{n + 1} \varphi\| 
 \leq 2^{-(n + 1)}.
\end{align*}
This implies that for each $m \in \N$, the sequence $(\varphi_{m,n})_{n}$ is Cauchy and hence convergent in $M_\ast$. Put $\Phi_{m} = \lim_{n \to \infty} \varphi_{m,n} \in M_\ast$ and observe that $\Phi_{m}$ is a normal state. We moreover have
\begin{align*}
\|\varphi - \Phi_{m}\| &\leq \|\varphi - \varphi_{m,n}\| + \|\varphi_{m,n} - \Phi_{m}\| \\
& \leq 2 \, \|\varphi p_{m} - p_{m} \varphi\| + \sum_{n \geq m} \|\varphi_{m,n} - \varphi_{m,n+1}\| \\
& \leq 2^{-m} + \sum_{n \geq m} 2^{-(n + 1)} 
= 2^{-(m - 1)}.
\end{align*}
This implies that $\lim_{m \to \infty} \Phi_{m} = \varphi$. Observe that $\Phi_m \, p_n = p_n \, \Phi_m$  for all $0 \leq m \leq n$.

We next claim that $\Phi_m$ is a {\em faithful} normal state for all $m \in \N$. Indeed, fix $m \in \N$ and let $x \in M^+$ such that $\Phi_m(x) = 0$. We prove by induction over $n \geq m$ that $\Phi_{n}(x) = 0$. By assumption, we have $\Phi_{m}(x) = 0$. Assume that $\Phi_{n}(x) = 0$ for some $n \geq m$. Observe that $0 = \Phi_{n}(x) = \Phi_{n + 1}(p_n x p_n + p_n^\perp x p_n^\perp)$. Denote by $q \in M$ the support of the normal state $\Phi_{n + 1}$. We have $qp_n x p_n q = 0 = qp_n^\perp x p_n^\perp q$. This implies that $x^{1/2} p_n q = 0 = x^{1/2} p_n^\perp q$ and hence $x^{1/2} q = 0$, that is, $qxq = 0$. Thus, $\Phi_{n + 1}(x) = 0$. Therefore, we have $\Phi_{n}(x) = 0$ for all $n \geq m$ and hence $\varphi(x) = \lim_{n\to \infty} \Phi_{n}(x) = 0$. Since $\varphi$ is faithful, we obtain $x = 0$. This shows that $\Phi_m \in M_\ast$ is faithful for every $m \in \N$. 

Letting $\psi := \Phi_0$, we have $p_n \in M^\psi$ for all $n \in \N$ and hence $p = (p_n)^\omega \in M' \cap (M^\psi)^\omega$. Since $\varphi^\omega(p) = 2^{-1}$, we have $p \neq 0, 1$. This implies that $M' \cap (M^\psi)^\omega$ is diffuse. Indeed, proceeding as in the proof of \cite[Corollary 3.8]{Co74}, let $f \in M' \cap (M^\psi)^\omega$ be any projection such that $\psi^\omega(f) = \lambda$ with $\lambda \neq 0, 1$. Write $f = (f_n)^\omega$ where $f_n \in M^\psi$ is a projection for every $n \in \N$. Observe that since $M$ is a factor, we have $\sigma\text{-weak} \lim_{n \to \omega} f_n = \lambda 1_M$. We can construct by induction an increasing sequence of integers $k_n \in \N$ satisfying the following properties:
\begin{itemize}
\item [(P1)] $|\psi(f_n f_{k_n}) - \lambda \psi(f_n)| \leq (n + 1)^{-1}$ for all $n \in \N$.
\item [(P2)] $\|f_n f_{k_n} - f_{k_n} f_n\|_\psi \leq (n + 1)^{-1}$ for all $n \in \N$.
\item [(P3)] $\|y_j f_{k_n} - f_{k_n} y_j\|_\psi \leq (n + 1)^{-1}$ for all $0 \leq j \leq n$.
\end{itemize}
It follows that $r := (f_n f_{k_n})^\omega \in M' \cap (M^\psi)^\omega$ is a projection satisfying $r \leq f$ and $\psi(r) = \lambda^2$. This shows that $f \in M' \cap (M^\psi)^\omega$ is not a minimal projection and hence $M' \cap (M^\psi)^\omega$ is diffuse.

(iii) $\Rightarrow$ (iv) The proof of this implication is entirely analogous to the one of \cite[Proposition 7]{Oz03} with $\mathcal N_0 = \mathcal M$ but we give the details for the sake of completeness. Fix a countable $\|\cdot\|_\psi$-dense subset $Y = \{y_n : n \in \N\} \subset M$. Since $M' \cap (M^\psi)^\omega$ is diffuse (note that $M^\psi$ is also diffuse), the proof of (ii) $\Rightarrow$ (iii) shows that we can construct by induction a sequence of projections $p_n \in M^\psi$ satisfying the following properties:
\begin{itemize}
\item[(P1)] $\psi(p_n) = 2^{-1}$ for all $n \in \N$.
\item [(P2)] $p_j p_n = p_n p_j$ for all $j, n \in \N$.
\item [(P3)] $\psi(p_{i_1} \cdots p_{i_r}) = 2^{-r}$ for all $r \geq 1$ and all $r$-tuples $(i_1, \dots, i_r)$ of pairwise distinct integers.
\item [(P4)] $\|y_j p_n - p_n y_j\|_\psi \leq 2^{-(n + 1)}$ for all $0 \leq j \leq n$.
\end{itemize}

For each $k \in \N$, define $D_k := \C p_k \oplus \C p_k^\perp$. For each pair $0 \leq m \leq n$, define $A_{m,n} := \bigvee_{m \leq k \leq n} D_k$ and $A_{m} := \bigvee_{m \leq k} D_k = \bigvee_{m \leq n} A_n$. Observe that $A_{m,n}, A_{m} \subset M^\psi$ for all $0 \leq m \leq n$. We also have that $(A_{m})_m$ is a decreasing sequence of diffuse abelian von Neumann subalgebras of $M^\psi$ by (P2) and (P3). Fix $j \in \N$ and let $n \geq m \geq j$. Whenever $C \subset M$ is a von Neumann subalgebra globally invariant under the modular automorphism group $\sigma^{\psi}$, denote by $\rE_C^{\psi} : M \to C$ the unique $\psi$-preserving conditional expectation. We have $A_{m,n+1} = A_{m,n} \vee D_{n + 1} \subset M^\psi$, $\rE_{(A_{m,n + 1})' \cap M}^{\psi} = \rE_{(A_{m,n})' \cap M}^{\psi} \circ \rE_{(D_{n + 1})' \cap M}^{\psi}$ (see e.g.\ \cite[Lemma 1.2.2]{Po83}) and
\begin{align*}
\left\|\rE_{(A_{m,n+1})' \cap M}^{\psi}(y_j) - \rE_{(A_{m,n})' \cap M}^{\psi}(y_j) \right\|_{\psi} &= \left \|\rE_{(A_{m,n})' \cap M}^{\psi} \left (\rE_{(D_{n + 1})' \cap M}^{\psi}(y_j) - y_j \right) \right \|_{\psi} \\
& \leq \left  \|\rE_{(D_{n + 1})' \cap M}^{\psi}(y_j) - y_j  \right \|_{\psi} \\
& \leq 2 \, \|p_{n + 1} y_j - y_j p_{n + 1}\|_{\psi} 
\leq 2^{-(n + 1)}.
\end{align*}
By \cite[Lemma 1.2 1$^\circ$]{Po81}, we have $\|y_j - \rE_{(A_{m})' \cap M}^{\psi}(y_j)\|_\psi = \lim_{n\to\infty} \|y_j - \rE_{(A_{m,n})' \cap M}^{\psi}(y_j)\|_\psi$ and hence
\begin{align*}
\left\|y_j - \rE_{(A_{m})' \cap M}^{\psi}(y_j) \right\|_\psi &= \lim_n \left\|y_j - \rE_{(A_{m,n})' \cap M}^{\psi}(y_j) \right\|_\psi \\
& \leq \left \| y_j - \rE_{(A_{m,n})' \cap M}^{\psi}(y_j) \right \|_\psi + \sum_{n \geq m} \left\|\rE_{(A_{m,n+1})' \cap M}^{\psi}(y_j) - \rE_{(A_{m,n})' \cap M}^{\psi}(y_j) \right\|_{\psi} \\
& \leq 2 \| p_{m}y_j - y_j p_{m} \|_\psi + \sum_{n \geq m} \left\|\rE_{(A_{m,n+1})' \cap M}^{\psi}(y_j) - \rE_{(A_{m,n})' \cap M}^{\psi}(y_j) \right\|_{\psi} \\
&\leq 2^{-m} + 2^{-m} = 2^{-(m - 1)}.
\end{align*}
It follows that $\lim_m \|y_j - \rE_{(A_{m})' \cap M}^{\psi}(y_j) \|_\psi = 0$ for all $j \in \N$. Since $Y \subset M$ is $\|\cdot\|_\psi$-dense, this implies that $\lim_m \|y - \rE_{(A_{m})' \cap M}^{\psi}(y)\|_\psi = 0$ for all $y \in M$ and hence $M = \bigvee_{n\in\N} ((A_n)' \cap M)$.

(iv) $\Rightarrow$ (i) For every $n \in \N$, choose a projection $p_n \in A_n \subset A_0$ such that $\varphi(p_n) = 2^{-1}$. Then $p := (p_n)^\omega \in M' \cap M^\omega$ and $\varphi^\omega(p) = 2^{-1}$. Therefore, $M' \cap M^\omega \neq \C1$ and hence $M' \cap M^\omega$ is diffuse since $M$ is a factor (see e.g.\ \cite[Corollary 2.6]{HR14}, or the final part of the proof of (ii) $\Rightarrow$ (iii)). 
\end{proof}

\section{Structure of AFP von Neumann algebras over arbitrary index sets}\label{structure}

In this section, we prove key results regarding the position of finite von Neumann subalgebras with expectation and with either nonamenable relative commutant (see Theorem \ref{thm-spectral-gap-general-AFP}) or nonamenable normalizer (see Theorem \ref{thm-normalizer-general-AFP}) inside arbitrary free product von Neumann algebras over arbitrary index sets.

\subsection*{Semifinite AFP von Neumann algebras over arbitrary index sets}

We will be using the following notation throughout this section.

\begin{nota}\label{notation-semifinite-AFP}
Let $I$ be any nonempty set and $(\mathcal B \subset \mathcal M_i)_{i \in I}$ any family of inclusions of semifinite $\sigma$-finite von Neumann algebras  with faithful normal conditional expectations $\rE_i : \mathcal M_i \to \mathcal B$, where $\mathcal B$ has a faithful normal semifinite trace $\Tr$ such that $\Tr\circ \rE_i$ is tracial on $\mathcal M_i$ for every $i \in I$. Assume moreover that $\mathcal B$ is amenable. Denote by $(\mathcal M, \rE) = \ast_{\mathcal B, i \in I} (\mathcal M_i, \rE_i)$ the corresponding semifinite amalgamated free product. For every nonempty subset $\mathcal G \subset I$, put $(\mathcal M_{\mathcal G}, \rE_{\mathcal G}) = \ast_{\mathcal B, i \in \mathcal G} (\mathcal M_i, \rE_i)$. By convention, put $\mathcal M_{\emptyset} := \mathcal B$. In this context, any trace means (an amplification $\Tr_n := \Tr \otimes \tr_n$ of) the trace $\Tr := \Tr\circ \rE$ or $\Tr\circ\rE_{\mathcal G}$.
\end{nota}

The next proposition will be used to reduce the problem of locating subalgebras inside arbitrary  semifinite amalgamated free product von Neumann algebras over {\em arbitrary} index sets to {\em finite} index sets.

\begin{prop}\label{proposition-infinite-AFP}
Keep Notation \ref{notation-semifinite-AFP}. Let $p \in \mathcal M$ be any nonzero finite trace projection and $\mathcal Q \subset p \mathcal M p$ any von Neumann subalgebra. Assume that for any nonzero projection $z \in \mathcal Q' \cap p \mathcal M p$ and any nonempty finite subset $\mathcal F \subset I$, we have $\mathcal Q z \preceq_{\mathcal M} \mathcal M_{\mathcal F^c}$. Then $\mathcal Q$ is amenable.
\end{prop}

\begin{proof}
The proof uses an idea due to Ioana. By contradiction, assume that $\mathcal Q$ is not amenable. Up to cutting down by a nonzero central projection in $\mathcal Z(\mathcal Q)$ if necessary, we may assume without loss of generality that $\mathcal Q$ has no amenable direct summand and that for any nonzero projection $z \in \mathcal Q' \cap p \mathcal M p$ and any nonempty finite subset $\mathcal F \subset I$, we have $\mathcal Q z \preceq_{\mathcal M} \mathcal M_{\mathcal F^c}$. By assumption and using \cite[Lemma 4.11]{HI15}, for every nonempty finite subset $\mathcal F \subset I$, there exist $n_{\mathcal F} \geq 1$, a finite trace projection $q_{\mathcal F} \in \mathbf M_{n_{\mathcal F}}(\mathcal M_{\mathcal F^c})$, a nonzero partial isometry $w_{\mathcal F} \in \mathbf M_{1, n_{\mathcal F}}(p\mathcal M)q_{\mathcal F}$ and a unital normal $\ast$-homomorphism $\pi_{\mathcal F} : \mathcal Q \to q_{\mathcal F}\mathbf M_{n_{\mathcal F}}(\mathcal M_{\mathcal F^c})q_{\mathcal F}$ such that $a w_{\mathcal F} = w_{\mathcal F} \pi_{\mathcal F}(a)$ for all $a \in \mathcal Q$ and $\lim_{\mathcal F} w_{\mathcal F}w_{\mathcal F}{}^* = p = 1_{\mathcal Q}$. Observe that $w_{\mathcal F}w_{\mathcal F}{}^* \in \mathcal Q' \cap p \mathcal M p$ and $w_{\mathcal F}{}^*w_{\mathcal F} \in \pi_{\mathcal F}(\mathcal Q)' \cap q_{\mathcal F}\mathbf M_{n_{\mathcal F}}(\mathcal M)q_{\mathcal F}$. Since $\pi_{\mathcal F}(\mathcal Q)$ has no amenable direct summand and $\mathcal B$ is amenable, we have $\pi_{\mathcal F}(\mathcal Q) \npreceq_{\mathbf M_{n_{\mathcal F}}(\mathcal M)} \mathbf M_{n_{\mathcal F}}(\mathcal B)$ and hence \cite[Theorem 2.5]{BHR12} shows $w_{\mathcal F}{}^*w_{\mathcal F} \in q_{\mathcal F}\mathbf M_{n_{\mathcal F}}(\mathcal M_{\mathcal F^c})q_{\mathcal F}$ for all $\mathcal F$. Thus, we may assume that $q_{\mathcal F} = w_{\mathcal F}{}^*w_{\mathcal F} \in \mathbf M_{n_{\mathcal F}}(\mathcal M_{\mathcal F^c})$ for all $\mathcal F$. It follows that $w_{\mathcal F}{}^* \mathcal Q w_{\mathcal F} \subset q_{\mathcal F}\mathbf M_{n_{\mathcal F}}(\mathcal M_{\mathcal F^c})q_{\mathcal F}$ for all $\mathcal F$.

Put $\widetilde {\mathcal M} := \mathcal M \ast_{\mathcal B} \mathcal M$, where we regard the left-hand copy of $\mathcal M$ as the original $\mathcal M$, and denote by $\Theta \in \Aut(\widetilde {\mathcal M})$ the free flip (trace preserving) automorphism. Likewise, for every $\mathcal F$, put $\widetilde {\mathcal M}_{\mathcal F} := \mathcal M_{\mathcal F} \ast_{\mathcal B} \mathcal M_{\mathcal F}$ and denote by $\Theta_{\mathcal F} \in \Aut(\widetilde {\mathcal M}_{\mathcal F})$ the free flip (trace preserving) automorphism. Regard $\Theta_{\mathcal F} \in \Aut(\widetilde {\mathcal M})$ by letting $\Theta_{\mathcal F}|_{\widetilde {\mathcal M}_{\mathcal F^c}} = \id_{\widetilde {\mathcal M}_{\mathcal F^c}}$ where $\widetilde {\mathcal M}_{\mathcal F^c} := \mathcal M_{\mathcal F^c} \ast_{\mathcal B} \mathcal M_{\mathcal F^c}$. We have $\lim_{\mathcal F} \Theta_{\mathcal F} = \Theta$ in $\Aut(\widetilde {\mathcal M})$. Observe that since $w_{\mathcal F}{}^* \mathcal Q w_{\mathcal F} \subset q_{\mathcal F}\mathbf M_{n_{\mathcal F}}(\mathcal M_{\mathcal F^c})q_{\mathcal F}$, we have $(\id_{n_{\mathcal F}} \otimes \Theta_{\mathcal F})(w_{\mathcal F}{}^* a w_{\mathcal F}) = w_{\mathcal F}{}^* a w_{\mathcal F}$ for all $a \in \mathcal Q$. Letting $\xi_{\mathcal F} := (\id_{n_{\mathcal F}} \otimes \Theta_{\mathcal F})(w_{\mathcal F}) w_{\mathcal F}{}^*$, for all $a \in \mathcal Q$, we have 
\begin{align*}
\Theta_{\mathcal F}(a) \, \xi_{\mathcal F} &= \Theta_{\mathcal F}(a) \, (\id_{n_{\mathcal F}} \otimes \Theta_{\mathcal F}) (w_{\mathcal F}) w_{\mathcal F}{}^* \\
&= (\id_{n_{\mathcal F}} \otimes \Theta_{\mathcal F}) (a w_{\mathcal F}) w_{\mathcal F}{}^* \\
&= (\id_{n_{\mathcal F}} \otimes \Theta_{\mathcal F}) (w_{\mathcal F} \, w_{\mathcal F}{}^*a w_{\mathcal F}) w_{\mathcal F}{}^* \\
&= (\id_{n_{\mathcal F}} \otimes \Theta_{\mathcal F}) (w_{\mathcal F}) \, (\id_{n_{\mathcal F}} \otimes \Theta_{\mathcal F})(w_{\mathcal F}{}^*a w_{\mathcal F}) \, w_{\mathcal F}{}^* \\
&= (\id_{n_{\mathcal F}} \otimes \Theta_{\mathcal F}) (w_{\mathcal F}) \, w_{\mathcal F}{}^*a w_{\mathcal F} \, w_{\mathcal F}{}^* \\
&= (\id_{n_{\mathcal F}} \otimes \Theta_{\mathcal F}) (w_{\mathcal F})w_{\mathcal F}{}^* \, a \\
&= \xi_{\mathcal F} \, a.
\end{align*}

Endow $\mathcal H := \rL^2(\widetilde{\mathcal M})$ with the $\mathcal M$-$\mathcal M$-bimodule structure given by $x \cdot \eta \cdot y := \Theta(x) \, \eta \, y$ for all $x, y \in \mathcal M$ and all $\eta \in \rL^2(\widetilde{\mathcal M})$. By construction and using \cite[Section 2]{Ue98a}, there exists a $\mathcal B$-$\mathcal B$-bimodule $\mathcal L$ such that we have the following isomorphism
$$\mathcal H \cong \rL^2(\mathcal M) \otimes_{\mathcal B} \mathcal L \otimes_{\mathcal B} \rL^2(\mathcal M)$$
as $\mathcal M$-$\mathcal M$-bimodules. (Indeed, for any amalgamated free product $(M,E) = (M_1,E_1)\ast_B (M_2,E_2)$ we have $\rL^2(M) \cong \rL^2(M_2)\otimes_B \mathcal K \otimes_B \rL^2(M_1)$ as $M_2$-$M_1$-bimodule with $\mathcal K := \rL^2(B)\oplus (\rL^2(M_1^\circ)\otimes_B\rL^2(M_2^\circ))\oplus\cdots\oplus(\rL^2(M_1^\circ)\otimes_B\cdots\otimes_B\rL^2(M_2^\circ))\oplus\cdots$.) Since $\mathcal B$ is amenable, \cite[Lemma 1.7]{AD93} shows that the $\mathcal M$-$\mathcal M$-bimodule $\mathcal H$ is weakly contained in the coarse $\mathcal M$-$\mathcal M$-bimodule $\rL^2(\mathcal M) \otimes \rL^2(\mathcal M)$. This implies (see the proof of \cite[Proposition 3.1]{CH08}) that the $p\mathcal Mp$-$p\mathcal Mp$-bimodule $p \cdot \mathcal H\cdot p$ is weakly contained in the coarse $p\mathcal Mp$-$p\mathcal Mp$-bimodule $\rL^2(p\mathcal M p) \otimes \rL^2(p \mathcal M p)$.

Regard $\xi_{\mathcal F} \in \mathcal H$ and put $\eta_{\mathcal F} := p \cdot \xi_{\mathcal F} \cdot p \in p \cdot \mathcal H \cdot p$. First, we have
\begin{align*}
\|\eta_{\mathcal F} - \xi_{\mathcal F}\|_2 
&= \|\Theta(p) \, \xi_{\mathcal F} - \xi_{\mathcal F} \|_2 \qquad\qquad\qquad \text{(since $\eta_{\mathcal F} = \Theta(p) \, \xi_{\mathcal F}$)} \\
&\leq \|(\Theta(p) - \Theta_{\mathcal F}(p)) \, \xi_{\mathcal F}\|_2 \qquad 
\text{(since $\Theta_{\mathcal F}(p) \, \xi_{\mathcal F} = \xi_{\mathcal F} \, p = \xi_{\mathcal F}$)} \\
& \leq \|\Theta(p) - \Theta_{\mathcal F}(p) \|_2 \, \|\xi_{\mathcal F}\|_\infty \\
& \leq \|\Theta(p) - \Theta_{\mathcal F}(p) \|_2 \to 0 \quad \text{as} \quad \mathcal F \to \infty.
\end{align*}
Then we have 
\begin{align*}
\|\xi_{\mathcal F}\|_2^2 &= \Tr(w_{\mathcal F} \, (\id_{n_{\mathcal F}} \otimes \Theta_{\mathcal F})(w_{\mathcal F}{}^* w_{\mathcal F}) \, w_{\mathcal F}{}^* ) \\
&= \Tr_{n_{\mathcal F}}((\id_{n_{\mathcal F}} \otimes \Theta_{\mathcal F})(w_{\mathcal F}{}^* w_{\mathcal F}) \, w_{\mathcal F}{}^*w_{\mathcal F} ) \\
&= \Tr_{n_{\mathcal F}}(w_{\mathcal F}{}^* w_{\mathcal F}) \qquad \text{(since $w_{\mathcal F}{}^*w_{\mathcal F} \in \mathbf M_{n_{\mathcal F}}(\widetilde{\mathcal M}_{\mathcal F^c})$)} \\
&= \Tr(w_{\mathcal F} w_{\mathcal F}{}^*) \to \Tr(p) \quad \text{as} \quad \mathcal F \to \infty.
\end{align*}
Since $\lim_{\mathcal F} \|\eta_{\mathcal F} - \xi_{\mathcal F}\|_2 = 0$, this implies that $\lim_{\mathcal F} \|\eta_{\mathcal F}\|_2 = \|p\|_2$. For all $x \in p\mathcal M p$ and all $\mathcal F$, we have
\begin{align*}
\|x \cdot \eta_{\mathcal F}\|_2 
= \|\Theta(x) \, \eta_{\mathcal F}\|_2 
\leq \|\Theta(x)\|_2 \, \|\eta_{\mathcal F}\|_\infty 
\leq \|x\|_2.
\end{align*}
For every $a \in \mathcal Q$, we have
\begin{align*}
\|a \cdot \xi_{\mathcal F} - \xi_{\mathcal F} \cdot a\|_2 
&= \|\Theta(a) \, \xi_{\mathcal F} - \xi_{\mathcal F} \, a\|_2 \\ 
&\leq \|(\Theta(a) - \Theta_{\mathcal F}(a)) \, \xi_{\mathcal F}\|_2 \qquad \text{(since $\Theta_{\mathcal F}(a) \, \xi_{\mathcal F} = \xi_{\mathcal F} \, a$)}\\
& \leq \|\Theta(a) - \Theta_{\mathcal F}(a)\|_2 \, \| \xi_{\mathcal F}\|_\infty \\
& \leq \|\Theta(a) - \Theta_{\mathcal F}(a)\|_2, 
\end{align*}
and hence $\lim_{\mathcal F} \| a \cdot \xi_{\mathcal F} -  \xi_{\mathcal F} \cdot a\|_2 = 0$. Since $\lim_{\mathcal F} \|\eta_{\mathcal F} - \xi_{\mathcal F}\|_2 = 0$, this implies that $\lim_{\mathcal F} \|a \cdot \eta_{\mathcal F} - \eta_{\mathcal F} \cdot a\|_2 = 0$ for all $a \in \mathcal Q$. By Connes's characterization of amenability \cite{Co75} applied to the finite von Neumann algebra $\mathcal Q$ and the net $(\eta_{\mathcal F})_{\mathcal F}$ in $p \cdot \mathcal H \cdot p$ (see also \cite[Lemma 2.3]{Io12}), it follows that $\mathcal Q$ has an amenable direct summand, a contradiction.
\end{proof}

\subsection*{Relative commutants inside AFP von Neumann algebras}

We begin by studying relative commutants inside {\em semifinite} amalgamated free product von Neumann algebras.

\begin{thm}\label{thm-spectral-gap-semifinite-AFP}
Keep Notation \ref{notation-semifinite-AFP}. Let $p \in \mathcal M$ be any nonzero finite trace projection and $\mathcal Q \subset p \mathcal M p$ any von Neumann subalgebra with no amenable direct summand and such that $\mathcal Q' \cap p \mathcal M p \npreceq_{\mathcal M} \mathcal B$. Then there exists $i \in I$ such that $\mathcal Q \preceq_{\mathcal M} \mathcal M_i$.
\end{thm}

\begin{proof}
Since $\mathcal Q' \cap p \mathcal M p \npreceq_{\mathcal M} \mathcal B$, we have $(\mathcal Q' \cap p \mathcal M p)^\omega \npreceq_{\mathcal M^\omega} \mathcal B^\omega$  by Proposition \ref{proposition-ultraproducts}. Since $(\mathcal Q' \cap p \mathcal M p)^\omega \subset \mathcal Q' \cap (p \mathcal M p)^\omega$, we also have $\mathcal Q' \cap (p \mathcal M p)^\omega \npreceq_{\mathcal M^\omega} \mathcal B^\omega$. For each nonempty finite subset $\mathcal F \subset I$, regard $\mathcal M = \mathcal M_{\mathcal F} \ast_{\mathcal B} \mathcal M_{\mathcal F^c}$. By \cite[Corollary 4.2]{HU15}, for every nonzero projection $z \in \mathcal Q' \cap p \mathcal M p$, we have $\mathcal Q z \preceq_{\mathcal M} \mathcal M_{\mathcal F}$ or $\mathcal Q z \preceq_{\mathcal M} \mathcal M_{\mathcal F^c}$. Since $\mathcal Q$ has no amenable direct summand, there exists a nonzero projection $z \in \mathcal Q' \cap p \mathcal M p$ and a nonempty finite subset $\mathcal F$ such that $\mathcal Q z \preceq_{\mathcal M} \mathcal M_{\mathcal F}$ by Proposition \ref{proposition-infinite-AFP}. Therefore, we have  $\mathcal Q \preceq_{\mathcal M} \mathcal M_{\mathcal F}$. Put $\mathcal P := \mathcal Q \vee (\mathcal Q' \cap p \mathcal M p)$. Since $\mathcal Q \npreceq_{\mathcal M} \mathcal B$, we also have that $\mathcal P \preceq_{\mathcal M} \mathcal M_{\mathcal F}$ by \cite[Proposition 2.7]{BHR12}.

Then there exist $n \geq 1$, a finite trace projection $q \in \mathbf M_n(\mathcal M_{\mathcal F})$, a nonzero partial isometry $w \in \mathbf M_{1, n}(p\mathcal M)q$ and a unital normal $\ast$-homomorphism $\pi : \mathcal P \to q\mathbf M_n(\mathcal M_{\mathcal F})q$ such that $a w = w \pi(a)$ for all $a \in  \mathcal P$. Observe that $ww^* \in \mathcal P' \cap p \mathcal M p = \mathcal Z(\mathcal P)$ and $w^*w \in \pi(\mathcal P)' \cap q\mathbf M_n(\mathcal M)q$. Since $\mathcal P$ has no amenable direct summand, we have $\pi( \mathcal P) \npreceq_{\mathbf M_n(\mathcal M_{\mathcal F})} \mathbf M_n(\mathcal B)$. This implies that $w^*w \in q\mathbf M_n(\mathcal M_{\mathcal F})q$ by \cite[Theorem 2.5]{BHR12} and hence we may assume that $q = w^*w$. We obtain $w^* \mathcal P w \subset q\mathbf M_n(\mathcal M_{\mathcal F})q$. Observe that $w^* \mathcal Q w$ and $w^* (\mathcal Q' \cap p \mathcal M p) w$ are commuting unital subalgebras of $w^*\mathcal Pw$ such that $w^* \mathcal Q w$ has no amenable direct summand and $w^* (\mathcal Q' \cap p \mathcal M p) w \npreceq_{\mathbf M_n(\mathcal M_{\mathcal F})} \mathbf M_n(\mathcal B)$ by Remark \ref{remark-intertwining} (2) (recall that $\mathcal Q' \cap p \mathcal M p \npreceq_{\mathcal M} \mathcal B$). Observe that for each $i \in \mathcal F$, $w^* \mathcal Q w \preceq_{\mathbf M_n(\mathcal M_{\mathcal F})} \mathbf M_n(\mathcal M_i)$ leads to $\mathcal Q \preceq_{\mathcal M} \mathcal M_i$ by Remark \ref{remark-intertwining} (2).  Therefore, we have showed that in order to prove Theorem \ref{thm-spectral-gap-semifinite-AFP}, we may assume that the index set $I$ is finite.

When the index set $I$ is finite, a straightforward induction procedure over $k := |I|$ using a combination of the above reasoning with \cite[Corollary 4.2]{HU15} proves the result. Indeed, assume that the result is true for any set $I$ such that $|I| = k$ with $k \geq 1$. Next, let $I$ be any set such that $|I| = k + 1$. Simply denote $I = \{1, \dots, k + 1\}$. Regard $\mathcal M = \mathcal M_{\mathcal F} \ast_{\mathcal B} \mathcal M_{k + 1}$ where $\mathcal F = \{1, \dots, k\}$. The same reasoning as in the first paragraph above shows that $\mathcal Q \preceq_{\mathcal M} \mathcal M_{\mathcal F}$ or $\mathcal Q \preceq_{\mathcal M} \mathcal M_{k + 1}$. If $\mathcal Q \preceq_{\mathcal M} \mathcal M_{k + 1}$, we are done. If $\mathcal Q \preceq_{\mathcal M} \mathcal M_{\mathcal F}$, the same reasoning as in the second paragraph above shows that with letting $\mathcal P := \mathcal Q \vee \mathcal (Q' \cap p \mathcal Mp)$ there exists $n \geq 1$, a finite trace projection $q \in \mathbf M_n(\mathcal M_{\mathcal F})$, a nonzero partial isometry $w \in \mathbf M_{1, n}(p\mathcal M)q$ and a unital normal $\ast$-homomorphism $\pi : \mathcal P \to q\mathbf M_n(\mathcal M_{\mathcal F})q$ such that $a w = w \pi(a)$ for all $a \in \mathcal P$. We may moreover assume that $w^*w = q$. Then $w^* \mathcal Q w$ and $w^* (\mathcal Q' \cap p \mathcal M p) w$ are commuting unital subalgebras of $w^*\mathcal Pw$ such that $w^* \mathcal Q w$ has no amenable direct summand and $w^* (\mathcal Q' \cap p \mathcal M p) w \npreceq_{\mathbf M_n(\mathcal M_{\mathcal F})} \mathbf M_n(\mathcal B)$. Using the induction hypothesis, there exists $i \in \mathcal F = \{1, \dots, k\}$ such that $w^* \mathcal Q w \preceq_{\mathbf M_n (\mathcal M_{\mathcal F})} \mathbf M_n(\mathcal M_i)$. Then $\mathcal Q \preceq_{\mathcal M} \mathcal M_i$ holds by Remark \ref{remark-intertwining} (2). This finishes the proof of the induction procedure and completes the proof of Theorem \ref{thm-spectral-gap-semifinite-AFP}.
\end{proof}

We now prove a general result locating finite subalgebras with expectation and with nonamenable relative commutant inside {\em arbitrary} amalgamated free product von Neumann algebras. This result will be used in the proof of the Main Theorem (Cases (i) and (ii)).

\begin{thm}\label{thm-spectral-gap-general-AFP}
Let $I$ be any nonempty set and $(B \subset M_i)_{i \in I}$ any family of inclusions of $\sigma$-finite von Neumann algebras with faithful normal conditional expectations $\rE_i : M_i \to B$. Assume moreover that $B$ is amenable. Denote by $(M, \rE) = \ast_{B, i \in I} (M_i, \rE_i)$ the corresponding amalgamated free product.

Let $1_A \in M$ be any nonzero projection and $A \subset 1_AM1_A$ any finite von Neumann subalgebra with expectation. Then at least one of the following conditions holds true:
\begin{itemize}
\item There exists $i \in I$ such that $A \preceq_{M} M_i$.
\item The von Neumann subalgebra $A' \cap 1_AM1_A$ is amenable.
\end{itemize}
\end{thm}

\begin{proof}
Put $\widetilde A = A \oplus \C(1_M - 1_A)$ and denote by $\rE_{\widetilde A} : M \to \widetilde A$ a faithful normal conditional expectation. Choose a faithful trace $\tau_{\widetilde A} \in \widetilde A_\ast$ and put $\psi = \tau_{\widetilde A} \circ \rE_{\widetilde A}$. Observe that $1_A \in M^\psi$ and the von Neumann subalgebras $A$ and $A' \cap 1_AM1_A$ are globally invariant under the modular automorphism group $\sigma^{\psi_{A}}$ of $\psi_{A} := \frac{\psi(1_A \, \cdot \, 1_A)}{\psi(1_A)}$. 

Assume that $A' \cap 1_AM1_A$ is not amenable. Observe that if $A \preceq_M B$, we are done. Hence we may further assume that $A \npreceq_M B$. Choose a nonzero central projection $z \in \mathcal Z(A' \cap 1_AM1_A)$ such that $(A' \cap 1_AM1_A)z$ has no amenable direct summand. Observe that $z \in M^\psi$ and the von Neumann subalgebras $Az$ and $(A' \cap 1_A M 1_A)z$ are globally invariant under the modular automorphism group $\sigma^{\psi_z}$ of $\psi_z := \frac{\psi(z \, \cdot \, z)}{\psi(z)}$. Then $\core_{\psi_z}((A' \cap 1_AM1_A)z)$ has no amenable direct summand by \cite[Proposition 2.8]{BHR12}. 

Fix a faithful state $\varphi \in B_\ast$ and put $\mathcal B := \core_\varphi(B)$, $\mathcal M := \core_{\varphi \circ \rE}(M)$ and $\mathcal M_i := \core_{\varphi \circ \rE_i}(M_i)$ for every $i \in I$. Let $q \in \rL_\psi(\R)$ be any nonzero finite trace projection and put $p := \Pi_{\varphi, \psi}(q)$ and $\mathcal Q := \Pi_{\varphi, \psi}(q \core_{\psi_z}((A' \cap 1_AM1_A)z) q)$. Then $\mathcal Q \subset p \mathcal M p$ has no amenable direct summand. Since $A \npreceq_M B$, we have $Az \npreceq_M B$ by \cite[Remark 4.2 (2)]{HI15}. By \cite[Lemma 2.4]{HU15} we obtain $\Pi_{\varphi, \psi}(\pi_{\psi}(Az) q) \npreceq_{\mathcal M} \mathcal B$. Since $\Pi_{\varphi, \psi}(\pi_{\psi}(Az) q) \subset \mathcal Q' \cap p \mathcal M p$, we conclude that $\mathcal Q' \cap p \mathcal M p \npreceq_{\mathcal M} \mathcal B$. 

By Theorem \ref{thm-spectral-gap-semifinite-AFP}, there exists $i \in I$ such that $\mathcal Q \preceq_{\mathcal M} \mathcal M_i$. Since $\mathcal Q \npreceq_{\mathcal M} \mathcal B$ (recall that $c_{\psi_z}((A' \cap 1_A M 1_A)z)$ has no amenable direct summand), we also have $\mathcal Q' \cap p \mathcal M p \preceq_{\mathcal M} \mathcal M_i$ by \cite[Proposition 2.7]{BHR12} and hence $\Pi_{\varphi, \psi}(\pi_{\psi}(Az) q) \preceq_{\mathcal M} \mathcal M_i$. By \cite[Lemma 2.4]{HU15}, this implies that $Az \preceq_M M_i$ and hence $A \preceq_M M_i$ by \cite[Remark 4.2 (2)]{HI15}.
\end{proof}

\subsection*{Normalizers inside AFP von Neumann algebras}

We begin by studying normalizers inside {\em semifinite} amalgamated free product von Neumann algebras. For a technical reason, we only deal with amalgamated free products of type ${\rm II_\infty}$ factors. This result will be sufficient for our purposes. 

\begin{thm}\label{thm-normalizer-semifinite-AFP}
Keep Notation \ref{notation-semifinite-AFP}. Assume moreover that $\mathcal B$ is a diffuse subalgebra and $\mathcal M_{\mathcal G}$ is a type ${\rm II_\infty}$ factor for every nonempty subset $\mathcal G \subset I$. Let $p \in \mathcal M$ be any nonzero finite trace projection and $\mathcal A \subset p \mathcal M p$ any amenable von Neumann subalgebra such that $\mathcal A  \npreceq_{\mathcal M} \mathcal B$. Put $\mathcal Q := \mathcal N_{p \mathcal M p}(\mathcal A)\dpr$ and assume that $\mathcal Q$ has no amenable direct summand. Then there exists $i \in I$ such that $\mathcal Q \preceq_{\mathcal M} \mathcal M_i$.
\end{thm}

\begin{proof}
For each nonempty finite subset $\mathcal F \subset I$, regard $\mathcal M = \mathcal M_{\mathcal F} \ast_{\mathcal B} \mathcal M_{\mathcal F^c}$. Let $z \in \mathcal Q' \cap p \mathcal M p$ be any nonzero projection. Since $\mathcal M$ is a type ${\rm II_\infty}$ factor and since $\mathcal B \subset \mathcal M$ is a diffuse subalgebra with trace preserving conditional expectation, there exists $u \in \mathcal U(\mathcal M)$ such that $uzu^* \in \mathcal B$. Since the unital inclusion $u \mathcal A z u^*\subset u\mathcal Q z u^*$ is regular  and since $u \mathcal A z u^* \npreceq_{\mathcal M} \mathcal B$ by assumption (and Remark \ref{remark-intertwining} (2)), we have $u\mathcal Q zu^* \preceq_{\mathcal M} \mathcal M_{\mathcal F}$ or $u\mathcal Q zu^* \preceq_{\mathcal M} \mathcal M_{\mathcal F^c}$ by Theorem~\ref{theorem-appendix-AFP} (together with the comment following it) and \cite[Proposition 2.7]{BHR12}. Accordingly, we have $\mathcal Q z \preceq_{\mathcal M} \mathcal M_{\mathcal F}$ or $\mathcal Q z \preceq_{\mathcal M} \mathcal M_{\mathcal F^c}$ (by Remark \ref{remark-intertwining} (2)). Since $\mathcal Q$ has no amenable direct summand, Proposition \ref{proposition-infinite-AFP} ensures that there exist a nonzero projection $z \in \mathcal Q' \cap p \mathcal M p$ and a nonempty finite subset $\mathcal F$ such that $\mathcal Q z \preceq_{\mathcal M} \mathcal M_{\mathcal F}$. Therefore, we have that $\mathcal Q \preceq_{\mathcal M} \mathcal M_{\mathcal F}$.

Then there exist $n \geq 1$, a finite trace projection $q \in \mathbf M_n(\mathcal M_{\mathcal F})$, a nonzero partial isometry $w \in \mathbf M_{1, n}(p\mathcal M)q$ and a unital normal $\ast$-homomorphism $\pi : \mathcal Q \to q\mathbf M_n(\mathcal M_{\mathcal F})q$ such that $a w = w \pi(a)$ for all $a \in  \mathcal Q$. Observe that $ww^* \in \mathcal Q' \cap p \mathcal M p$ and $w^*w \in \pi(\mathcal Q)' \cap q\mathbf M_n(\mathcal M)q$. Since $\mathcal Q$ has no amenable direct summand, we have $\pi( \mathcal Q) \npreceq_{\mathbf M_n(\mathcal M_{\mathcal F})} \mathbf M_n(\mathcal B)$. Then \cite[Theorem 2.5]{BHR12} implies that $w^*w \in q\mathbf M_n(\mathcal M_{\mathcal F})q$ and hence we may assume that $q = w^*w$. We obtain $w^* \mathcal Q w \subset q\mathbf M_n(\mathcal M_{\mathcal F})q$.

Since $ww^* \in \mathcal Q' \cap p \mathcal M p$, it follows that the unital inclusion $w^* \mathcal A w \subset w^* \mathcal Q w$ is regular, $w^* \mathcal Q w$ has no amenable direct summand, and $w^* \mathcal A w \npreceq_{\mathbf M_n(\mathcal M_{\mathcal F})} \mathbf M_n(\mathcal B)$ (by Remark \ref{remark-intertwining} (2) and $\mathcal A \npreceq_{\mathcal M} \mathcal B$). By Remark \ref{remark-intertwining} (2), $w^* \mathcal Q w \preceq_{\mathbf M_n(\mathcal M_{\mathcal F})} \mathbf M_n(\mathcal M_i)$ implies that $\mathcal Q \preceq_{\mathcal M} \mathcal M_i$  for every $i \in \mathcal F$. Therefore, since $\mathbf M_n(\mathcal M_{\mathcal G})$ is a type ${\rm II_\infty}$ factor for every nonempty subset $\mathcal G \subset \mathcal F$ and since $\mathbf M_n(\mathcal B)$ is diffuse, we have showed that in order to prove Theorem \ref{thm-spectral-gap-semifinite-AFP}, we may assume that the index set $I$ is finite.

When the index set $I$ is finite, a straightforward induction procedure over $k := |I|$ using a combination of the above reasoning with the assumptions that $\mathcal M_{\mathcal G}$ is a type ${\rm II_\infty}$ factor for every nonempty subset $\mathcal G \subset I$ and $\mathcal B$ is diffuse and with Theorem \ref{theorem-appendix-AFP} proves the result (see the last paragraph of the proof of Theorem \ref{thm-spectral-gap-semifinite-AFP}). 
\end{proof}

We now prove a result locating certain finite subalgebras with expectation and with nonamenable normalizer inside {\em arbitrary}  free products of $\sigma$-finite factors. This result will be used in the proof of the Main Theorem (Case (iii)).

\begin{thm}\label{thm-normalizer-general-AFP}
Let $I$ be any nonempty set and $(M_i, \varphi_i)_{i \in I}$ any family of $\sigma$-finite factors endowed with any faithful normal states. Denote by $(M, \varphi) = \ast_{i \in I} (M_i, \varphi_i)$ the corresponding free product.

Let $1_A \in M$ be any nonzero projection and $A \subset 1_AM1_A$ be any amenable finite von Neumann subalgebra with expectation such that $A' \cap 1_AM1_A = \mathcal Z(A)$. Then at least one of the following conditions holds true:
\begin{itemize}
\item There exists $i \in I$ such that $A \preceq_{M} M_i$.
\item The von Neumann subalgebra $\mathcal N_{1_AM1_A}(A)\dpr$ is amenable.
\end{itemize}
\end{thm}

\begin{proof}
Denote by $R_\infty$ the unique type ${\rm III_1}$ AFD factor. Put $\widetilde B = \C1_M \ovt R_\infty$, $\widetilde M = M \ovt R_\infty$ and $\widetilde \rE = \varphi \otimes \id_{R_\infty}$, $\widetilde{M}_i = M_i \ovt R_\infty$ and $\widetilde \rE_i = \varphi_i \otimes \id_{R_\infty}$ for every $i \in I$. We may and will naturally regard the pair $(\widetilde M, \widetilde \rE)$ as 
$$(\widetilde M, \widetilde \rE) = \ast_{\widetilde B, i \in I} (\widetilde M_i, \widetilde \rE_i).$$
Observe that $\widetilde M_i = M_i \ovt R_\infty$ is a type ${\rm III_1}$ factor for every $i \in I$. (This is well-known without explicit reference and can be confirmed by computing the (smooth) flow of weights; see \cite[Corollary 6.8]{CT76}.) For every nonempty subset $\mathcal G \subset I$,  $M_{\mathcal G}$ is a factor by \cite[Theorem 4.1]{Ue10}, and hence $\widetilde M_{\mathcal G} = M_{\mathcal G} \ovt R_\infty$ is a type ${\rm III_1}$ factor by the same reasoning as above. 

Fix an irreducible type ${\rm II_1}$ subfactor $R \subset R_\infty$ with expectation (whose existence is explained in e.g.~\cite[Example 1.6]{Ha85}). Put $\widetilde A = (A \oplus \C(1_M - 1_A)) \ovt R$ and denote by $\rE_{\widetilde A} : \widetilde M \to \widetilde A$ a faithful normal conditional expectation. Choose a faithful  trace $\tau_{\widetilde A} \in \widetilde A_\ast$ and put $\psi = \tau_{\widetilde A} \circ \rE_{\widetilde A}$. We will simply denote $D := (1_A \otimes 1_R) \widetilde A (1_A \otimes 1_{R})$ and $1_{D} := 1_A \otimes 1_R$. Observe that $D' \cap 1_{D} \widetilde M 1_{D} = \mathcal Z(D)$, the unital inclusion $(\mathcal N_{1_A M 1_A}(A)\dpr \oplus \C(1_M - 1_A) )\ovt \C1_R  \subset \widetilde M$ is with expectation and also so is 
$$
 \mathcal N_{1_A M 1_A}(A)\dpr \ovt \C1_R = 1_D ((\mathcal N_{1_A M 1_A}(A)\dpr \oplus \C(1_M - 1_A)) \ovt \C1_R ) 1_D \subset \mathcal N_{1_D \widetilde M 1_D}(D)\dpr.
$$
Moreover, we have $1_D \in \widetilde M^\psi$ and the von Neumann subalgebras $D$ and $\mathcal N_{1_D\widetilde M1_D}(D)\dpr$ are globally invariant under the modular automorphism group $\sigma^{\psi_{D}}$ of $\psi_{D} := \frac{\psi(1_D \, \cdot \, 1_D)}{\psi(1_D)}$.

Observe that we have 
$$\core_{\psi_{D}}(\mathcal N_{1_D\widetilde M1_D}(D)\dpr) \subset \mathcal N_{\core_{\psi_{D}}(1_D\widetilde M1_D)}(\core_{\psi_{D}}(D))\dpr.$$ Indeed, let $u \in \mathcal N_{1_D\widetilde M1_D}(D)$ and $t \in \R$. For every $a \in D$, we have
$$u \sigma_t^{\psi_{D}}(u^*) \, a = u \, \sigma_t^{\psi_{D}}(u^* a u) \, \sigma_t^{\psi_{D}}(u)^* = u \, u^* a u \, \sigma_t^{\psi_{D}}(u)^* = a \, u \sigma_t^{\psi_{D}}(u)^*.$$
This shows that $u \sigma_t^{\psi_{D}}(u)^* \in D' \cap 1_D\widetilde M1_D = \mathcal Z(D)$ and hence we have $\pi_{\psi_{D}}(u) \lambda_{\psi_{D}}(t) \pi_{\psi_{D}}(u)^* \in \pi_{\psi_{D}}(\mathcal Z(D)) \lambda_{\psi_{D}}(t)$. Therefore we obtain that $\core_{\psi_{D}}(\mathcal N_{1_D\widetilde M1_D}(D)\dpr) \subset \mathcal N_{\core_{\psi_{D}}(1_D\widetilde M1_D)}(\core_{\psi_{D}}(D))\dpr$ and this inclusion is with trace preserving conditional expectation. 

 Assume that $\mathcal N_{1_AM1_A}(A)\dpr$ is not amenable. Observe that if $A \preceq_M B$, we are done. Hence we may further assume that $A \npreceq_M B$. Then $\mathcal N_{1_D\widetilde M1_D}(D)\dpr$ is not amenable (since it contains $\mathcal N_{1_A M 1_A}(A)\dpr \ovt \C1_R$ with expectation) and $D \npreceq_{\widetilde M} \widetilde B$ by \cite[Lemma 4.6]{HI15}. Choose a nonzero projection $z \in \mathcal Z(\mathcal N_{1_D\widetilde M1_D}(D)\dpr)$ such that $\mathcal N_{1_D\widetilde M1_D}(D)\dpr z$ has no amenable direct summand. Observe that $z \in \widetilde M^\psi$, $Dz \subset z \widetilde M^\psi z$ and $\mathcal N_{z \widetilde Mz}(Dz)\dpr = \mathcal N_{1_D \widetilde M1_D}(D)\dpr z$ is globally invariant under the modular automorphism group $\sigma^{\psi_z}$ of $\psi_z = \frac{\psi(z \, \cdot \, z)}{\psi(z)}$. Then $\core_{\psi_z}(\mathcal N_{z\widetilde Mz}(Dz)\dpr)$ has no amenable direct summand by \cite[Proposition 2.8]{BHR12}. 

Fix a faithful state $\chi \in (R_\infty)_\ast$ and put $\mathcal B := \core_\chi(\widetilde B)$, $\mathcal M := \core_{\varphi \otimes \chi}(\widetilde M)$ and $\mathcal M_i := \core_{\varphi_i \otimes \chi}(\widetilde M_i)$ for every $i \in I$. Observe that $\mathcal B \subset \mathcal M$ is a diffuse subalgebra with trace preserving conditional expectation and $\mathcal M_{\mathcal G}$ is a type ${\rm II_\infty}$ factor for every nonempty subset $\mathcal G \subset I$. Let $q \in \rL_\psi(\R)$ be any nonzero finite trace projection. Put $p := \Pi_{\varphi \otimes \chi, \psi}(q) \in \mathcal M$, $\mathcal A := \Pi_{\varphi  \otimes \chi, \psi}( \core_{\psi_z}(Dz) q)$ and $\mathcal Q := \Pi_{\varphi  \otimes \chi, \psi}(q \core_{\psi_z}(\mathcal N_{z\widetilde Mz}(Dz)\dpr) q)$.
Since $Dz \subset z \widetilde M^\psi z$ and $\rL(\R)$ is a MASA in $\mathbf B(\rL^2(\R))$, we have $\core_{\psi_z}(D z)' \cap \core_{\psi_z}(z\widetilde Mz) = \mathcal Z(\core_{\psi_z}(D z))$ and hence
$$ q\big(\mathcal N_{\core_{\psi_z}(z\widetilde Mz)}(\core_{\psi_z}(D z))\dpr \big)q = \mathcal N_{q\core_{\psi_z}(z\widetilde Mz)q}(\core_{\psi_z}(Dz) q)\dpr$$
by Proposition \ref{prop-corner-normalizer}. Then we have $\mathcal Q = \mathcal N_{p \mathcal M p}(\mathcal A)\dpr$, $\mathcal Q$ has no amenable direct summand, and $\mathcal A \npreceq_{\mathcal M} \mathcal B$ by \cite[Lemma 2.4]{HU15} since $Dz \npreceq_{\widetilde M} \widetilde B$. By Theorem \ref{thm-normalizer-semifinite-AFP}, there exists $i \in I$ such that $\mathcal A \preceq_{\mathcal M} \mathcal M_i$. Then \cite[Lemma 2.4]{HU15} shows that $Dz \preceq_{\widetilde M} \widetilde M_i$ and hence $D \preceq_{\widetilde M} \widetilde M_i$. Finally, \cite[Lemma 4.6]{HI15} guarantees $A \preceq_M M_i$.
\end{proof}

We point out that when dealing with {\em tracial} free product von Neumann algebras, Theorem \ref{thm-normalizer-general-AFP} holds true for any family $(M_i, \tau_i)_{i \in I}$ of tracial von Neumann algebras and any amenable von  Neumann subalgebra $A \subset 1_A M 1_A$.

\subsection*{Relative property (T) subalgebras inside AFP von Neumann algebras} Recall from \cite[Definition 4.2.1]{Po01} that an inclusion of tracial von Neumann algebras $A \subset (N, \tau)$ is said to have {\em relative property (T)} if for every net $(\Phi_i : N \to N)_{i \in I}$ of subtracial subunital completely positive maps such that $\lim_i \|\Phi_i (x) - x\|_2 = 0$ for all $x \in N$, we have
$$\lim_{i} \sup_{y \in \Ball(A)} \|\Phi_i(y) - y\|_2 = 0.$$

We begin by locating relative property (T) subalgebras inside {\em semifinite} amalgamated free product von Neumann algebras. This is a semifinite analogue of \cite[Theorem 4.3]{IPP05}.

\begin{thm}\label{thm-relative-property-T-semifinite-AFP}
Keep Notation \ref{notation-semifinite-AFP}. Let $p \in \mathcal M$ be any nonzero finite trace projection and $\mathcal A \subset p\mathcal M p$ any von Neumann subalgebra with relative property (T). Then there exists $i \in I$ such that $\mathcal A \preceq_{\mathcal M} \mathcal M_i$.
\end{thm}

\begin{proof}
For each nonempty finite subset $\mathcal F \subset I$, regard $\mathcal M = \mathcal M_{\mathcal F} \ast_{\mathcal B} \mathcal M_{\mathcal F^c}$ and denote by $\rE_{\mathcal M_{\mathcal F}} : \mathcal M \to \mathcal M_{\mathcal F}$ the unique trace preserving conditional expectation. Define the net $\Phi_{\mathcal F} : p \mathcal M p \to  p\mathcal Mp$ of subtracial subunital unital completely positive maps by $\Phi_{\mathcal F}(x) = p \rE_{\mathcal M_{\mathcal F}}(x) p$ for all $x \in p\mathcal M p$. Observe that $\lim_{\mathcal F} \|\Phi_{\mathcal F}(x) - x\|_2 = 0$ for all $x \in p\mathcal M p$.

By relative property (T) of the inclusion $\mathcal A \subset p \mathcal M p$, there exists a nonempty finite subset $\mathcal F \subset I$ such that 
$$\|\rE_{\mathcal M_{\mathcal F}}(u)\|_2 \geq \|p \rE_{\mathcal M_{\mathcal F}}(u) p\|_2 = \|\Phi_{\mathcal F}(u)\|_2 \geq \frac12 \|p\|_2, \forall u \in \mathcal U(\mathcal A).$$
 If $\mathcal A \npreceq_{\mathcal M} \mathcal M_{\mathcal F}$, then by Theorem \ref{theorem-intertwining}, there exists a net $(u_j)_{j\in J}$ in $\mathcal U(\mathcal A)$ with $\lim_j \| \rE_{\mathcal M_{\mathcal F}}(u_j)\|_2 = 0$; contradicting the above inequality. Hence we obtain $\mathcal A \preceq_{\mathcal M} \mathcal M_{\mathcal F}$. 

If $\mathcal A \preceq_{\mathcal M} \mathcal B$, then $\mathcal A \preceq_{\mathcal M} \mathcal M_i$ for any $i \in I$ and we are done. If $\mathcal A \npreceq_{\mathcal M} \mathcal B$, then Theorem \ref{theorem-intertwining} and Lemma \ref{lemma-intertwining} altogether enable us to choose $n \geq 1$, a finite trace projection $q \in \mathbf M_n(\mathcal M_{\mathcal F})$, a nonzero partial isometry $w \in \mathbf M_{1, n}(p\mathcal M)q$ and a unital normal $\ast$-homomorphism $\pi : \mathcal A \to q\mathbf M_n(\mathcal M_{\mathcal F})q$ such that $\pi(\mathcal A) \npreceq_{\mathbf M_n(\mathcal M_{\mathcal F})} \mathbf M_n(\mathcal B)$ and $a w = w \pi(a)$ for all $a \in  \mathcal A$. By \cite[Theorem 2.5]{BHR12}, we have $\pi(\mathcal A)' \cap q\mathbf M_n(\mathcal M)q = \pi(\mathcal A)' \cap q\mathbf M_n(\mathcal M_{\mathcal F})q$. Then we have $w^*w \in \pi(\mathcal A)' \cap q\mathbf M_n(\mathcal M_{\mathcal F})q$ and so we  may assume that $q = w^*w$ and $w^* \mathcal A w = \pi(\mathcal A) \subset q\mathbf M_n(\mathcal M_{\mathcal F})q$.

By relative property (T) of the inclusion $\mathcal A \subset p \mathcal M p$ and \cite[Proposition 4.7]{Po01}, the unital inclusion $w^* \mathcal A w \subset q\mathbf M_n(\mathcal M)q$ has relative property (T). 
Consider
$$\mathbf M_n(\mathcal M) = \left( \ast_{\mathbf M_n(\mathcal B), i \in \mathcal F} \mathbf M_n(\mathcal M_i) \right ) \ast_{\mathbf M_n(\mathcal B)} \mathbf M_n(\mathcal M_{\mathcal F^c}).$$
Since $\mathcal A \npreceq_{\mathcal M} \mathcal B$, we have $w^* \mathcal A w \npreceq_{\mathbf M_n(\mathcal M)} \mathbf M_n(\mathcal B)$ by Remark \ref{remark-intertwining} (2). Since moreover $w^* \mathcal A w \subset q\mathbf M_n(\mathcal M_{\mathcal F})q$, we have $w^* \mathcal A w \npreceq_{\mathbf M_n(\mathcal M)} \mathbf M_n(\mathcal M_{\mathcal F^c})$ by Theorem \ref{theorem-intertwining} with the help of Lemma \ref{lemma-control-sequence}.  
Since the unital inclusion $w^* \mathcal A w \subset q\mathbf M_n(\mathcal M)q$ moreover has relative property (T), \cite[Theorem 3.3]{BHR12}  (whose proof works well for semifinite amalgamated free products of finitely many algebras) shows that there exists $i \in \mathcal F$ such that $w^* \mathcal A w \preceq_{\mathbf M_n(\mathcal M)} \mathbf M_n(\mathcal M_i)$. By Remark \ref{remark-intertwining} (2), this implies that $\mathcal A \preceq_{\mathcal M} \mathcal M_i$.
\end{proof}

We point out that we do not need to assume $\mathcal B$ to be amenable in Theorem \ref{thm-relative-property-T-semifinite-AFP}. We finally deduce the following result that will be used in the proof of the Main Theorem (Case (iv)).

\begin{thm}\label{thm-relative-property-T-general-AFP}
Let $I$ be any nonempty set and $(B \subset M_i)_{i \in I}$ any family of inclusions of $\sigma$-finite von Neumann algebras with faithful normal conditional expectations $\rE_i : M_i \to B$. Denote by $(M, \rE) = \ast_{B, i \in I} (M_i, \rE_i)$ the corresponding amalgamated free product.

Let $1_Q \in M$ be any nonzero projection and $Q \subset 1_QM1_Q$ any finite von Neumann subalgebra with expectation that possesses a von Neumann subalgebra $A \subset Q$ with relative property (T). Then there exists $i \in I$ such that $A \preceq_M M_i$.
\end{thm}

\begin{proof}
Put $\widetilde Q = Q \oplus \C (1_M - 1_Q)$ and let $\rE_{\widetilde Q} : M \to \widetilde Q$ be a faithful normal conditional expectation. Choose a faithful trace $\tau_{\widetilde Q} \in (\widetilde Q)_\ast$ and put $\psi = \tau_{\widetilde Q} \circ \rE_{\widetilde Q}$. Observe that $1_Q \in M^\psi$ and $Q \subset 1_Q M^\psi 1_Q$. 

Fix a faithful state $\varphi \in B_\ast$ and put $\mathcal B := \core_\varphi(B)$, $\mathcal M := \core_{\varphi \circ \rE}(M)$ and $\mathcal M_i := \core_{\varphi \circ \rE_i}(M_i)$ for every $i \in I$. Fix a nonzero finite trace projection $q \in \rL_\psi(\R)$ and put $p := \Pi_{\varphi, \psi}(q) \in \mathcal M$ and $\mathcal A := \Pi_{\varphi, \psi}(\pi_\psi(A) q) \subset p\mathcal Mp$. Since the inclusion $A \subset Q$ has relative property (T) and since 
$$\left( \Pi_{\varphi, \psi}(\pi_\psi(A) q) \subset \Pi_{\varphi, \psi}(\pi_\psi(Q) q) \right ) \cong (A \subset Q)$$ 
and
$$\mathcal A = \Pi_{\varphi, \psi}(\pi_\psi(A) q) \subset \Pi_{\varphi, \psi}(\pi_\psi(Q) q) \subset p \mathcal M p,$$
the inclusion $\mathcal A \subset p \mathcal M p$ also has relative property (T). By Theorem \ref{thm-relative-property-T-semifinite-AFP}, there exists $i \in I$ such that $\mathcal A \preceq_{\mathcal M} \mathcal M_i$, and hence $A \preceq_M M_i$ by \cite[Lemma 2.4]{HU15}.
\end{proof}

\section{Proof of the Main Theorem}\label{proof}

Assume that $M$ and $N$ are isomorphic and identify $M = N$. Note however that we cannot identify $\varphi$ with $\psi$.

Fix an arbitrary $i \in I$. We first prove the following intermediate assertion: 

($\diamondsuit$) There exist $j = \alpha(i) \in J$, $n_j \geq 1$ and a nonzero partial isometry $v_j \in \mathbf M_{1, n_j}(M)$ such that $v_j^*v_j \in \mathbf M_{n_j}(N_j)$, $v_jv_j^* \in M_i$ and $v_j^* M_i v_j \subset v_j^*v_j \mathbf M_{n_j}(N_j) v_j^*v_j$. Observe that the unital inclusion $v_jv_j^* M_i v_jv_j^* \subset v_j v_j^* M v_j v_j^*$ is with expectation and hence so is the unital inclusion $v_j^* M_i v_j \subset v_j^* v_j \mathbf M_{n_j}(M) v_j^* v_j$. Therefore the unital inclusion $v_j^* M_i v_j \subset v_j^*v_j \mathbf M_{n_j}(N_j) v_j^*v_j$ is with expectation. When $N_j$ is semifinite, we will be able to choose the partial isometry $v_j \in \mathbf M_{1, n_j}(M)$ in such a way that $v_j^* v_j \in \mathbf M_{n_j}(N_j)$ has finite trace. This is because we are going to use Theorem \ref{theorem-intertwining} and show, as a crucial step, that $A \preceq_M N_j$ for a well-chosen {\em finite} von Neumann subalgebra $A \subset M_i$ with expectation. 

We will treat cases (i), (ii), (iii), (iv) separately as follows.  

{\bf Case (i).} Assume that $M_i$ is not prime, and hence we may write $M_i = P_1 \ovt P_2$ with diffuse factors $P_1$ and $P_2$. We may assume without loss of generality that $P_2$ is not amenable. Choose a faithful state $\chi_1 \in (P_1)_\ast$ such that $(P_1)^{\chi_1}$ is diffuse (see \cite[Theorem 11.1]{HS90}) and a faithful state $\chi_2 \in (P_2)_\ast$, and put $\chi = \chi_1 \otimes \chi_2$. Observe that there exists a diffuse abelian von Neumann subalgebra $A_1 \subset (P_1)^{\chi_1}$, and put $A = A_1 \otimes \C1$. Since $P_2$ is not amenable and since the unital inclusion $\C 1 \otimes P_2 \subset A' \cap M$ is with expectation (observe that $\C1 \otimes P_2 \subset M$ is with expectation), $A' \cap M$ is not amenable either. By Theorem \ref{thm-spectral-gap-general-AFP}, there exists $j = \alpha(i) \in J$ such that $A \preceq_M N_j$.

There exist $n_j \geq 1$, a projection $q_j \in \mathbf M_{n_j}(N_j)$, a nonzero partial isometry $v_j \in \mathbf M_{1, n_j}(M)$ and a unital normal $\ast$-homomorphism $\pi : A \to q_j\mathbf M_{n_j}(N_j)q_j$ such that the unital inclusion $\pi(A) \subset q_j \mathbf M_{n_j}(N_j) q_j$ is with expectation and $a v_j = v_j \pi(a)$ for all $a \in A$. By Proposition \ref{prop-with-expectation} (see Remark \ref{remark-intertwining} (1)), the unital inclusions $Av_jv_j^* \subset v_jv_j^* M v_jv_j^*$  and $v_j^* A v_j = \pi(A) v_j^*v_j \subset v_j^*v_j \mathbf M_{n_j}(M) v_j^*v_j$ are with expectation. By \cite[Proposition 2.7 (2)]{HU15}, we have $v_j v_j^* \in A' \cap M = A' \cap M_i$, $v_j^*v_j \in \mathbf M_{n_j}(N_j)$ and $v_j^* (A' \cap M_i) v_j \subset v_j^*v_j \mathbf M_{n_j}(N_j) v_j^*v_j$.

Observe that $A' \cap M_i = ((A_1)' \cap P_1) \ovt P_2$. By the same reasoning as in the proof of \cite[Lemma 4.13]{HI15} and by Lemma \ref{lemma-equivalence}, there exist nonzero projections $p_1 \in (A_1)' \cap (P_1)^{\chi_1}$ and $p_2 \in (P_2)^{\chi_2}$ such that $p_1p_2 \precsim v_jv_j^*$ in $A' \cap M_i$. Let $u \in A' \cap M_i$ be a partial isometry such that $uu^* = p_1p_2$ and $u^*u \leq v_jv_j^*$. We have $a \, u v_j = u v_j \, \pi(a)$ for all $a \in A$ and $(uv_j) (uv_j)^* = u v_j v_j^* u^* = p_1p_2$ and $(uv_j)^* (uv_j) = v_j^* u^*u v_j \in \mathbf M_{n_j}(N_j)$. So, up to replacing $v_j$ with $u v_j$, we may assume that $v_j v_j^* = p_1p_2$.

Observe that the unital inclusion $p_2P_2p_2p_1 \subset p_1p_2 M p_1 p_2$ is with expectation and so is the unital inclusion $v_j^* P_2 p_1 v_j \subset v_j^*v_j \mathbf M_{n_j}(N_j) v_j^*v_j$ (recall that $v_j^* (A' \cap M_i) v_j \subset v_j^*v_j \mathbf M_{n_j}(N_j) v_j^*v_j$). Then \cite[Proposition 2.7 (1)]{HU15} shows that 
\begin{align*}
v_j^* P_1 p_2 v_j &= (v_j^* P_2 p_1 v_j)' \cap (v_j^* M v_j) \\
&= (v_j^* P_2 p_1 v_j)' \cap v_j^*v_j \mathbf M_{n_j}(M) v_j^*v_j \\
&= (v_j^* P_2 p_1 v_j)' \cap v_j^*v_j \mathbf M_{n_j}(N_j) v_j^*v_j \\
&\subset v_j^*v_j \mathbf M_{n_j}(N_j) v_j^*v_j.
\end{align*}
Since  $v_j^* M_i v_j = v_j^* P_1 p_2 v_j \vee v_j^* P_2 p_1 v_j$, we obtain $v_j^* M_i v_j \subset v_j^*v_j \mathbf M_{n_j}(N_j) v_j^*v_j$.

{\bf Case (ii).} Assume that $M_i$ has property Gamma, that is, the central sequence algebra $(M_i)' \cap (M_i)^\omega$ is diffuse. By Theorem \ref{theorem-structure-gamma}, there exists a decreasing sequence $(A_n)_n$ of diffuse abelian subalgebras of $M_i$ with expectation such that $M_i = \bigvee_n ((A_n)' \cap M_i)$. Since $M_i$ is not amenable, there exists $n \in \N$ such that $(A_n)' \cap M_i$ is not amenable. Observe that $(A_n)' \cap M_i = (A_n)' \cap M$ by \cite[Proposition 2.7 (1)]{HU15}. By Theorem \ref{thm-spectral-gap-general-AFP}, there exists $j = \alpha(i) \in J$ such that $A_n \preceq_M N_j$.

There exist $n_j \geq 1$, a projection $q_j \in \mathbf M_{n_j}(N_j)$, a nonzero partial isometry $v_j \in \mathbf M_{1, n_j}(M)$ and a unital normal $\ast$-homomorphism $\pi : A_n \to  q_j\mathbf M_{n_j}(N_j) q_j$ such that the unital inclusion $\pi(A_n) \subset q_j \mathbf M_{n_j}(N_j) q_j$ is with expectation and $a v_j = v_j \pi(a)$ for all $a \in A_n$. By \cite[Proposition 2.7 (2)]{HU15}, we have $v_j v_j^* \in (A_n)' \cap M = (A_n)' \cap M_i$, $v_j^*v_j \in \mathbf M_{n_j}(N_j)$ and $v_j^* ((A_n)' \cap M_i) v_j \subset v_j^*v_j \mathbf M_{n_j}(N_j) v_j^*v_j$. Observe that $\pi(A_k) \subset \pi(A_n)$ is with expectation for every $k \geq n$ (since $A_n$ is abelian). Hence, the inclusion $\pi(A_k) \subset q_j\mathbf M_{n_j}(N_j)q_j$ is with expectation for every $k \geq n$. As in the second paragraph in case (i) we observe that $v_j^* ((A_k)' \cap M_i) v_j \subset v_j^*v_j \mathbf M_{n_j}(N_j) v_j^*v_j$ for every $k \geq n$. Since $M_i = \bigvee_{n \in \N} ((A_n)' \cap M_i)$, we finally obtain $v_j^* M_i v_j \subset v_j^*v_j \mathbf M_{n_j}(N_j) v_j^*v_j$.

{\bf Case (iii).} Assume that $M_i$ possesses an amenable finite von Neumann subalgebra $A$ with expectation such that $A' \cap M_i = \mathcal Z(A)$ and $\mathcal N_{M_i}(A)\dpr = M_i$. Since $M_i$ is not a type ${\rm I}$ factor, it is easy to see that $A$ is necessarily diffuse and hence \cite[Proposition 2.7]{HU15} shows that $A' \cap M = A' \cap M_i = \mathcal Z(A)$ and $\mathcal N_M(A)\dpr = \mathcal N_{M_i}(A)\dpr = M_i$. By Theorem \ref{thm-normalizer-general-AFP}, there exists $j = \alpha(i) \in J$ such that $A \preceq_M N_j$. Namely, there exist $n_j \geq 1$, a projection $q_j \in \mathbf M_{n_j}(N_j)$, a nonzero partial isometry $v_j \in \mathbf M_{1, n_j}(M)$ and a unital normal $\ast$-homomorphism $\pi : A \to q_j\mathbf M_{n_j}(N_j)q_j$ such that the unital inclusion $\pi(A) \subset q_j\mathbf M_{n_j}(N_j)q_j$ is with expectation and $a v_j = v_j \pi(a)$ for all $a \in A$. By \cite[Proposition 2.7]{HU15}, we have $v_j v_j^* \in A' \cap M = A' \cap M_i$ and $v_j^*v_j \in \mathbf M_{n_j}(N_j)$, and hence $v_j^* M_i v_j \subset v_j^*v_j \mathbf M_{n_j}(N_j) v_j^*v_j$.

{\bf Case (iv).} Assume that $M_i$ is a ${\rm II_1}$ factor that possesses a regular diffuse von Neumann subalgebra $A \subset M_i$ with relative property (T). By Theorem \ref{thm-relative-property-T-general-AFP}, there exists $j = \alpha(i) \in J$ such that $A \preceq_M N_j$. In the exactly same way as in the proof of case (iii), we conclude that there exist $n_j \geq 1$ and  a nonzero partial isometry $v_j \in \mathbf M_{1, n_j}(M)$ such that $v_j^*v_j \in \mathbf M_{n_j}(N_j)$, $v_jv_j^* \in M_i$ and $v_j^* M_i v_j \subset v_j^*v_j \mathbf M_{n_j}(N_j) v_j^*v_j$.

We have completed the proof of the desired intermediate assertion ($\diamondsuit$).  

By symmetry, for any given $j \in J$, there exist $i = \beta(j) \in I$, $m_i \geq 1$ and a nonzero partial isometry $w_i \in \mathbf{M}_{1,m_i}(M)$ such that $w_i^* w_i \in \mathbf{M}_{m_i}(M_i)$, $w_i w_i^* \in N_j$ and $w_i^* N_j w_i \subset w_i^* w_i \mathbf{M}_{m_i}(M_i) w_i^* w_i$. Moreover, the unital inclusion $w_i^* N_j w_i \subset w_i^* w_i \mathbf{M}_{m_i}(M_i) w_i^* w_i$ is with expectation. 

For every $i \in I$, put $w_{i}^{(n_{\alpha(i)})} := w_{i} \otimes 1_{n_{\alpha(i)}} \in \mathbf M_{1, m_i} (M)\otimes \mathbf M_{n_{\alpha(i)}}(\C)= \mathbf{M}_{n_{\alpha(i)}, n_{\alpha(i)} m_i}(M)$. Observe that $w_{i}^{(n_{\alpha(i)})}\big(w_{i}^{(n_{\alpha(i)})}\big)^* = w_i w_i^* \otimes 1_{n_{\alpha(i)}} \in \mathbf M_{n_{\alpha(i)}}(N_{\alpha(i)})$, $\big(w_{i}^{(n_{\alpha(i)})}\big)^* w_{i}^{(n_{\alpha(i)})} = w_i^*w_i \otimes 1_{n_{\alpha(i)}} \in \mathbf{M}_{n_{\alpha(i)} m_{i}}(M_{\beta(\alpha(i))})$ and
\begin{align*}
\big(w_{i}^{(n_{\alpha(i)})}\big)^* \mathbf{M}_{n_{\alpha(i)}}(N_{\alpha(i)}) w_{i}^{(n_{\alpha(i)})} &= w_i^* N_{\alpha(i)} w_i \otimes \mathbf M_{n_{\alpha(i)}}(\C) \\
& \subset w_i^*w_i \mathbf M_{m_i}(M_{\beta(\alpha(i))})w_i^*w_i \otimes \mathbf M_{n_{\alpha(i)}}(\C) \\
&= \big(w_{i}^{(n_{\alpha(i)})}\big)^* w_{i}^{(n_{\alpha(i)})} \mathbf{M}_{n_{\alpha(i)} m_{i}}(M_{\beta(\alpha(i))}) \big(w_{i}^{(n_{\alpha(i)})}\big)^* w_{i}^{(n_{\alpha(i)})}. 
\end{align*}
Since the inclusion $w_i^* N_{\alpha(i)} w_i \subset w_i^*w_i \mathbf M_{m_i}(M_{\beta(\alpha(i))})w_i^*w_i$ is with expectation, so is the above inclusion.

Since $M_i$ and $N_{\alpha(i)}$ are diffuse factors and since the projection $(v_{\alpha(i)})^* v_{\alpha(i)} \in \mathbf M_{n_{\alpha(i)}}(N_{\alpha(i)})$ has finite trace if $N_{\alpha(i)}$ is semifinite as claimed in the first paragraph of the proof, up to shrinking  $v_{\alpha(i)} (v_{\alpha(i)})^* \in M_{i}$ if necessary, we may further choose the partial isometry $v_{\alpha(i)} \in \mathbf M_{1, n_{\alpha(i)}}(M)$ so that $(v_{\alpha(i)})^* v_{\alpha(i)} \precsim w_{i}^{(n_{\alpha(i)})}\big(w_{i}^{(n_{\alpha(i)})}\big)^*$ in $\mathbf M_{n_{\alpha(i)}}(N_{\alpha(i)})$. Since $N_{\alpha(i)}$ is a factor, we can find a nonzero partial isometry $u \in \mathbf{M}_{n_{\alpha(i)}}(N_{\alpha(i)})$ such that $u u^* = (v_{\alpha(i)})^* v_{\alpha(i)}$ and $u^*u \leq w_{i}^{(n_{\alpha(i)})}\big(w_{i}^{(n_{\alpha(i)})}\big)^*$. Then $v := v_{\alpha(i)}u w_{i}^{n_{\alpha(i)}}$ is a nonzero partial isometry in $\mathbf{M}_{1, n_{\alpha(i)} m_{i}}(M)$ such that 
\begin{align*}
vv^* 
&= v_{\alpha(i)}u w_{i}^{(n_{\alpha(i)})}\big(w_{i}^{(n_{\alpha(i)})}\big)^* u^* (v_{\alpha(i)})^* 
= v_{\alpha(i)} uu^* (v_{\alpha(i)})^* 
= v_{\alpha(i)} (v_{\alpha(i)})^* \in M_i \\
v^*v 
&= \big(w_{i}^{(n_{\alpha(i)})}\big)^* u^* (v_{\alpha(i)})^* v_{\alpha(i)}u w_{i}^{(n_{\alpha(i)})} 
= \big(w_{i}^{(n_{\alpha(i)})}\big)^* u^* u w_{i}^{(n_{\alpha(i)})} \in \mathbf{M}_{n_{\alpha(i)} m_{i}}(M_{\beta(\alpha(i))})
\end{align*} 
and
\begin{align}\label{Eq1}
v^* M_i v &= \big(w_{i}^{(n_{\alpha(i)})}\big)^* u^* (v_{\alpha(i)})^* \, M_i \, v_{\alpha(i)}u\,w_{i}^{(n_{\alpha(i)})} \\ \nonumber
&\subset \big(w_{i}^{(n_{\alpha(i)})}\big)^* \, u^* \,  \mathbf M_{n_{\alpha(i)}}(N_{\alpha(i)}) \, u \, w_{i}^{(n_{\alpha(i)})} \\ \nonumber
& \subset v^* v\mathbf{M}_{n_{\alpha(i)} m_{i}}(M_{\beta(\alpha(i))}) v^*v. 
\end{align}
Note that the inclusions in \eqref{Eq1} are with expectation.

By Lemma \ref{lemma-control}, we have $\beta(\alpha(i)) = i$ for every $i \in I$. Since the inclusions in \eqref{Eq1} are with expectation and since $vv^* \in M_i$ and $v^*v \in \mathbf{M}_{n_{\alpha(i)} m_{i}}(M_{i}) $, we necessarily  have $v \in \mathbf M_{1,n_{\alpha(i)} m_{i}}(M_i)$ by \cite[Proposition 2.7 (1)]{HU15}. Therefore, \eqref{Eq1} must be equality with $\beta(\alpha(i))=i$. This implies that $v_{\alpha(i)} u = v \big(w_{i}^{(n_{\alpha(i)})}\big)^* \in \mathbf M_{1, n_{\alpha(i)}}(M)$ with $(v_{\alpha(i)} u) (v_{\alpha(i)} u)^* = v \big(w_{i}^{(n_{\alpha(i)})}\big)^* w_{i}^{(n_{\alpha(i)})} v^* \in M_i$, $(v_{\alpha(i)} u)^* (v_{\alpha(i)} u) = u^*u \in \mathbf M_{n_{\alpha(i)}}(N_{\alpha(i)})$ and $u^* (v_{\alpha(i)})^* M_i v_{\alpha(i)} u = u^* u \mathbf{M}_{n_{\alpha(i)}}(N_{\alpha(i)})u^* u$. By symmetry, we have $\alpha(\beta(j)) = j$ for every $j \in J$. This shows that $\alpha : I \to J$ is indeed a bijection and $M_i$ and $N_{\alpha(i)}$ are stably isomorphic to each other for every $i \in I$. Hence we have proved item (1) of the Main Theorem.

Assume moreover that $M_i$ is a type ${\rm III}$ factor for every $i \in I$. This forces $N_j$ to be a type ${\rm III}$ factor for every $j \in J$. Therefore, up to conjugating by partial isometries in $M_i$ and $N_{\alpha(i)}$, we may assume that $n_{\alpha(i)} = 1$ and that there exists a unitary $u_i \in \mathcal U(M)$ such that $u_i M_i u_i^* = N_{\alpha(i)}$ for every $i \in I$. The uniqueness of the bijection $\alpha : I \to J$ as in item (2) of the Main Theorem is guaranteed by Lemma \ref{lemma-control}. Therefore, we have completed the proof of the Main Theorem.

\section{Further results}\label{further-results}

Following \cite{Oz03, VV05}, we say that a $\sigma$-finite diffuse von Neumann algebra $M$ is {\em solid} if for any diffuse von Neumann subalgebra $A \subset M$ with expectation, the relative commutant $A' \cap M$ is amenable. More generally, we will say that a $\sigma$-finite (not necessarily diffuse) von Neumann algebra $M$ is solid if either $M$ is atomic or if its nonzero diffuse direct summand is solid. Recall that whenever $M$ is a diffuse solid von Neumann algebra, $p\mathbf M_n(M)p$ is also solid for all $n \geq 1$ and all nonzero projections $p \in \mathbf M_n(M)$ (see e.g.\ \cite[Proposition 3.2]{HR14} for a similar statement and its proof). The class of solid von Neumann algebras includes bi-exact group von Neumann algebras \cite{BO08, Oz03}, free quantum group von Neumann algebras \cite{VV05} and free Araki-Woods factors \cite{Ho07}.

Part of the technology provided for proving the Main Theorem also enables us to prove the following characterization of solidity for free products with respect to arbitrary faithful normal states and over arbitrary index sets. It moreover generalizes the main result of \cite{GJ07}.

\begin{thm}\label{theorem-solidity}
Let $I$ be any nonempty set and $(M_i, \varphi_i)_{i \in I}$ be any family of von Neumann algebras endowed with any faithful normal states. Then, for the corresponding free product $(M, \varphi) = \ast_{i \in I} (M_i, \varphi_i)$, the free product von Neumann algebra $M$ is solid if and only if so are all $M_i$.
\end{thm}

\begin{proof} 
 (The only if part) Assume that some $M_i$ is not solid. By definition, there exist a nonzero projection $z \in \mathcal Z(M_i)$ and a diffuse von Neumann subalgebra $P \subset M_i z$ with expectation such that the relative commutant $P' \cap M_i z$ is nonamenable. Since the unital inclusion $P' \cap M_i z \subset zM_i z$ is with expectation so is the unital inclusion $P' \cap M_i z \subset z M z$. This implies that the unital inclusion $P' \cap M_i z \subset P' \cap zMz$ is with expectation and hence $P' \cap zMz$ is nonamenable. Therefore, $zMz$ is not solid and neither is $M$.  

(The if part) Assume that all $M_i$ are solid. Suppose on the contrary that $M$ is not solid. Then there exist a diffuse von Neumann subalgebra $Q \subset 1_Q M 1_Q$ with expectation such that the relative commutant $Q' \cap 1_Q M 1_Q$ is nonamenable. As in the proof of \cite[Lemma 2.1]{HU15}, choose a faithful state $\psi \in M_\ast$ such that $1_Q \in M^\psi$, $Q$ is globally invariant under the modular automorphism group $\sigma^{\psi_Q}$ where $\psi_Q = \frac{\psi(1_Q \cdot 1_Q)}{\psi(1_Q)}$ and $A := Q^{\psi_Q}$ is diffuse. Since $Q' \cap 1_Q M 1_Q \subset A' \cap 1_Q M 1_Q$ with $1_Q = 1_A$ is with expectation, $A' \cap 1_A M 1_A$ is also nonamenable. Up to cutting down by a suitable nonzero central projection $z \in \mathcal Z(A' \cap 1_A M 1_A)$, for which $(A' \cap 1_A M 1_A)z$ has no amenable direct summand and up to replacing $A$ with $Az$ (note that $z \in M^\psi$ and $Az \subset zMz$ is with expectation), we may further assume without loss of generality that the relative commutant $A' \cap 1_A M 1_A$ has no amenable direct summand. By Theorem \ref{thm-spectral-gap-general-AFP}, there exists $i \in I$ such that $A \preceq_M M_i$.

Then there exist $n \geq 1$, a projection $q \in \mathbf M_n(M_i)$, a nonzero partial isometry $w \in \mathbf M_{1, n}(1_A M)q$ and a unital normal $\ast$-homomorphism $\pi : A \to q\mathbf M_n(M_i)q$ such that the unital inclusion $\pi(A) \subset q\mathbf M_n(M_i)q$ is with expectation and $aw = w \pi(a)$ for all $a \in A$. By Remark \ref{remark-intertwining} (1) both of the inclusions $A ww^* \subset ww^* M ww^*$ and $\pi(A) w^*w \subset w^*w \mathbf M_n(M) w^*w$ are with expectation. Proceeding as in the proof of the Main Theorem (case (i)), we have $w^*w \in q\mathbf M_n(M_i)q$ and $w^* A w$ and $w^*(A' \cap 1_A M 1_A)w$ are commuting subalgebras of $w^*w\mathbf M_n(M_i)w^*w$ with expectation. Since $w^* A w$ is diffuse and $w^*(A' \cap 1_A M 1_A)w$ is not amenable, $w^*w\mathbf M_n(M_i)w^*w$ is not solid. This however contradicts the fact that $M_i$ is solid.
\end{proof}

The first part of the above proof actually shows that any von Neumann subalgebra of a solid von Neumann algebra with expectation must be solid. 

\begin{rem}
Recall that a tracial von Neumann algebra $M$ is {\em strongly solid} if for any amenable diffuse von Neumann subalgebra $A \subset M$, the normalizer $\mathcal N_M(A)\dpr$ is amenable. Using a combination of the proofs of Theorem \ref{theorem-solidity} and \cite[Theorem 1.8]{Io12} with Theorem \ref{thm-normalizer-general-AFP} (for tracial von Neumann algebras; see the remark after its proof) in place of Theorem \ref{thm-spectral-gap-general-AFP}, we can also show that a given tracial free product von Neumann algebra over an {\em arbitrary} index set is strongly solid if and only if so are all the component algebras. 

We point out that we can then obtain examples of strongly solid ${\rm II_1}$ factors that do not have the weak$^\ast$ completely bounded approximation property (CBAP). Indeed, for every $n \geq 1$, take a lattice $\Gamma_n < \Sp(n, 1)$ and denote by $(M, \tau) = \ast_{n \in \N\setminus \{0\}} (\rL(\Gamma_n), \tau_{\Gamma_n})$ the canonical tracial free product ${\rm II_1}$ factor. By \cite[Theorem B]{CS11} and the above fact, $M$ is a strongly solid ${\rm II_1}$ factor. Moreover, it follows from \cite{CH88} that $M$ does not have the weak$^\ast$ CBAP.
\end{rem}

\begin{rem}
Any diffuse solid von Neumann algebra $M$ with property Gamma (and with separable predual) is necessarily amenable. Indeed, by Theorem \ref{theorem-structure-gamma}, there exists a decreasing sequence $(A_n)_n$ of diffuse abelian von Neumann subalgebras of $M$ with expectation such that $M = \bigvee_{n \in\N}((A_n)' \cap M)$. By solidity, $(A_n)' \cap M$ is amenable for every $n \in \N$ and hence $M$ is amenable. 
\end{rem}

\appendix

\section{Normalizers inside semifinite AFP von Neumann algebras}\label{appendix}

\subsection*{Ozawa--Popa's relative amenability in the semifinite setting}

Let $(M, \Tr)$ be any semifinite $\sigma$-finite von Neumann algebra endowed with a faithful normal semifinite trace and $B \subset M$ any von Neumann subalgebra with trace preserving conditional expectation $\rE_B : M \to B$. Denote by $\langle M, B \rangle$ the basic extension associated with $\rE_B$ and by $e_B$ the canonical Jones projection. Then there exists a faithful normal semifinite operator-valued weight, called the {\em dual} operator-valued weight, $\widehat{\rE}_B : \langle M,B\rangle_+ \to \widehat{M}_+$ satisfying $\widehat{\rE}_B(e_B) = 1$ (see e.g.~\cite[\S2.1]{ILP96}). Moreover, the linear span of $Me_B M$ forms a $\sigma$-strongly dense $\ast$-subalgebra of $\langle M, B \rangle$ and $\sigma_t^{\Tr\circ\widehat{\rE}_B}(e_B) = e_B$ for all $t \in \mathbf{R}$. Thus, $\Tr_{ \langle M,B\rangle}:= \Tr\circ\widehat{\rE}_B$ becomes a faithful normal semifinite trace on $\langle M, B\rangle$. 

\begin{thm}\label{T1} {\rm(\cite[Theorem 2.1]{OP07})} Let $p \in M$ be any nonzero projection with $\Tr(p) < +\infty$ and $A \subset pMp$ any von Neumann subalgebra. Write $\tau := \frac{1}{\Tr(p)}\Tr|_{pMp}$. The following conditions are equivalent: 
\begin{enumerate}
\item There exists an $A$-central state $\varphi$ on $p\langle M, B\rangle p$ such that $\varphi|_{pMp} = \tau$. 
\item There exists an $A$-central state $\varphi$ on $p\langle M, B\rangle p$ such that $\varphi|_{pMp}$ is normal and such that $\varphi|_{\mathcal{Z}(A'\cap pMp)}$ is faithful. 
\item There exists a conditional expectation $\Phi : p\langle M, B\rangle p \to A$ such that $\Phi|_{pMp}$ gives  the unique $\tau$-preserving conditional expectation from $pMp$ onto $A$. 
\item There exists a net $(\xi_i)_{i\in I}$ of vectors in $\rL^2(\langle M, B\rangle, \Tr_{\langle M, B\rangle})$ such that 
\begin{itemize}
\item $p\xi_i p = \xi_i$ for all $i \in I$,
\item $\lim_i \langle x\xi_i, \xi_i \rangle_{\Tr_{\langle M, B\rangle}} = \tau(x)$ for all $x \in pMp$ and
\item  $\lim_i \Vert a\xi_i - \xi_i a\Vert_{2, \Tr_{\langle M, B\rangle}} = 0$ for all $a \in A$. 
\end{itemize}
\end{enumerate}
We will say that \emph{$A$ is amenable relative to $B$ inside $M$} if one of the above equivalent conditions holds. 
\end{thm}
\begin{proof}
Observe that we have a natural identification of $\rL^2(p\langle M,B\rangle p, \Tr_{\langle M, B\rangle}|_{p\langle M,B\rangle p})$ with $p\cdot \rL^2(\langle M, B\rangle,\Tr_{\langle M, B\rangle}) \cdot p$ as $pMp$-$pMp$-bimodules. Then the proof of \cite[Theorem 2.1]{OP07} applies {\em mutatis mutandis}.  
\end{proof}

\begin{lem}\label{L2} {\rm(\cite[Corollary 2.3]{OP07} and \cite[Lemma 2.3]{Io12})} Let $p \in M$ be any nonzero projection with $\Tr(p) < +\infty$ and $A \subset pMp$ any von Neumann subalgebra. Let $\mathcal{L}$ be any $B$-$M$-bimodule. Assume that there exists a net $(\xi_i)_{i \in I}$ of vectors in $p\mathcal{H}$ with $\mathcal{H} := \rL^2(M,\Tr)\otimes_B\mathcal{L}$ such that the following conditions hold:
\begin{itemize}
\item $\limsup_i \Vert x\xi_i\Vert_{\mathcal{H}} \leq \Vert x\Vert_{2, \tau}$ for all $x \in pMp$, 
\item $\limsup_i \Vert\xi_i\Vert_{\mathcal{H}} > 0$ and 
\item $\lim_i \Vert a\xi_i - \xi_i a\Vert_{\mathcal{H}} = 0$ for all $a \in A$. 
\end{itemize}
Then there exists a nonzero projection $z \in \mathcal{Z}(A' \cap pMp)$ such that $Az$ is amenable relative to $B$ inside $M$.
\end{lem}
\begin{proof}
Observe that $\langle M,B \rangle = (J^M B J^M)' \cap \mathbf B(\rL^2(M))$ also acts naturally on $\mathcal{H}$ in this semifinite setting, where $J^M$ is the modular conjugation on the standard form $\rL^2(M)$. Then the proof of \cite[Lemma 2.3]{Io12} applies {\em mutatis mutandis} to obtain item (2) in Theorem \ref{T1}.   
\end{proof}

\subsection*{Vaes's dichotomy result in the semifinite setting} 
 
Let $(M, \rE) = (M_1, \rE_1) \ast_B (M_2, \rE_2)$ be any {\em semifinite} amalgamated free product von Neumann algebra endowed with a faithful normal semifinite trace $\Tr$ such that $\Tr\circ \rE = \Tr$. Let $q \in B$ be any nonzero projection such that $\Tr(q) < +\infty$. Up to replacing $\Tr$ with $\frac{1}{\Tr(q)}\Tr$ if necessary, we may and will assume that $\Tr(q) = 1$.

Denote by $\mathbf{F}_2 = \langle \gamma_1,\gamma_2 \rangle$  the free group on two generators and put 
\begin{align*}
(\widetilde{M},\widetilde{\rE}) &= (M, \rE) \ast_B (B \ovt \rL(\mathbf{F}_2),\id \otimes \tau_{\mathbf{F}_2}), \\
(\widetilde{M}_i,\widetilde{\rE}_i) &= (M_i, \rE_i) \ast_B (B\ovt \rL(\langle\gamma_i\rangle),\id \otimes\tau_{\langle \gamma_i\rangle}), \quad i\in \{1,2\}.
\end{align*}
Denote by $\mathbf{F}_2  \to \rL(\mathbf{F}_2) : \gamma \mapsto \lambda_\gamma$ the canonical unitary representation and regard $\rL(\F_2) \cong \mathbf{C}1_B\otimes \rL(\mathbf{F}_2) \subset \widetilde{M}$. Then we can naturally identify $(\widetilde{M},\widetilde{\rE}) = (\widetilde{M}_1,\widetilde{\rE}_1) \ast_B (\widetilde{M}_2,\widetilde{\rE}_2)$. Following \cite[\S2]{IPP05}, we can construct a $1$-parameter unitary group $u_i^t$ in $\rL(\langle \gamma_i\rangle) \subset \widetilde{M}_i \subset \widetilde{M}$ such that $u_i^1 = \lambda_{\gamma_i}$ and $\tau_{\langle \gamma_i \rangle}(u_i^t) = \frac{\sin(\pi t)}{\pi t}$ for all $t \in \R$. 

 Fix an arbitrary faithful state $\chi \in B_\ast$. Then $\sigma_t^{\chi\circ\widetilde{\rE}_i} = \sigma_t^{\chi\circ \rE_i} \ast (\sigma_t^\chi \otimes \id)$ (see \cite[Theorem 2.6]{Ue98a}) and hence $u_i^t$ lies in the centralizer of $\chi\circ\widetilde{\rE}_i$ for all $t \in \R$. Therefore, we have $\chi\circ\widetilde{\rE}_i = (\chi\circ\widetilde{\rE}_i)\circ\Ad(u_i^t)$, implying that $\widetilde{\rE}_i = \widetilde{\rE}_i\circ\Ad(u_i^t)$ for all $t \in \R$. Consequently, $\theta_t := \Ad(u_1^t) \ast \Ad(u_2^t) \in \Aut(\widetilde{M})$ is well-defined for all $t \in \R$. A similar consideration shows that $\sigma_t^{\Tr_B\circ \widetilde{\rE}} = \sigma_t^{\Tr} \ast (\sigma_t^{\Tr_B} \otimes \id) = \id$ with $\Tr_B := \Tr|_B$ so that $\widetilde{\Tr} := \Tr_B\circ\widetilde{\rE}$ gives a faithful normal semifinite trace on $\widetilde{M}$ extending $\Tr$ naturally. The triple $(M \subset \widetilde{M}, \widetilde{\Tr}, \theta_t)$ is the semifinite analogue of Popa's malleable deformation for {\em tracial} amalgamated free product von Neumann algebras as defined in \cite[\S2]{IPP05}. The basic inequalities such as \cite[Eq.(3.1),(3.2)]{Va13} hold true as they are (see e.g.\ \cite[\S3.1]{BHR12} with the necessary refinement along \cite[\S3.1]{Va13}). Observe that $\theta_t(q) = q$ for every $t \in \mathbf{R}$ so that $\theta_t(p) \leq q$ for every projection $p \in qMq$ and every $t \in \R$.

Recall that the key observation of \cite{Io12} is that the von Neumann algebra $N := \bigvee_{\gamma \in \mathbf{F}_2} \lambda_{\gamma}M\lambda_{\gamma}^*$ is identified with the amalgamated free product of infinitely many copies of $(M,\rE)$ over $\mathbf F_2$ as index set and that $\widetilde{M}$ admits the crossed product decomposition $\widetilde{M} = N \rtimes_{\Ad(\lambda)}\mathbf{F}_2$ whose canonical conditional expectation is denoted by $\rE_N : \widetilde M \to N$. Moreover, $q\widetilde{M}q = qNq\rtimes_{\Ad(\lambda)}\mathbf{F}_2$ holds naturally and the canonical conditional expectation $\rE_{qNq} : q\widetilde{M}q \to qNq$ coincides with the restriction of $\rE_N : \widetilde{M} \to N$ to $q N q$ since $q \in B \subset N \subset \widetilde{M}$ and thus $[\lambda_{\gamma},q] = 0$ for all $\gamma \in \mathbf{F}_2$. 

\begin{thm}[{\cite[Theorem 3.2]{Va13}}]\label{T3} Let $p \in qMq$ be any nonzero projection and $A \subset pMp$ any  von Neumann subalgebra. Assume that for all $t \in (0,1)$, $\theta_t(A)$ is amenable relative to $qNq$ inside $q\widetilde{M}q$. Then at least one of the following conditions holds: 
\begin{itemize}
\item Either $A \preceq_M M_1$ or $A \preceq_M M_2$ holds. 
\item $A$ is amenable relative to $B$ inside $M$. 
\end{itemize}
\end{thm}
\begin{proof} The proof is  identical to the one of \cite[Theorem 3.2]{Va13} with only minor modifications. This is why we will only sketch it. The most essential part of Vaes's proof is done at the Hilbert space level and hence it suffices to explain how to provide the right framework to modify the proof accordingly.

The functional $\tau := \widetilde{\Tr}|_{q\widetilde{M}q}$ defines a faithful normal tracial state on $q\widetilde{M}q = qNq\rtimes_{\Ad(\lambda)}\mathbf{F}_2$, since $\Tr(q) = 1$. Denote by $\langle q\widetilde{M}q, qNq \rangle$ the basic extension of $q\widetilde{M}q$ by $\rE_{qNq} : q \widetilde M q \to q N q$ with Jones projection $e_{qNq}$. To simplify the notation, we will simply write $\Tr := \tau\circ\widehat{\rE}_{qNq}$ where 
$$\widehat{\rE}_{qNq} : \langle q\widetilde{M}q, qNq \rangle_+ \to \widehat{q\widetilde{M}q}_+$$ is the dual faithful normal semifinite operator-valued weight satisfying $\widehat{\rE}_{qNq}(e_{qNq}) = 1_{qNq} = q$.

Let $I$ be the set of all the quadruplets $i = (X,Y,\delta,t)$ with finite subsets $X \subset q\widetilde{M}q$ and $Y \subset \mathcal U(A)$, $0 < \delta <1$ and $0 < t < 1$. The set $I$ becomes a directed set with the order relation $(X,Y,\delta,t) \leq (X',Y',\delta',t')$ defined by $X \subset X'$, $Y \subset Y'$, $\delta \geq \delta'$ and $t \geq t'$. Since $\theta_t(A)$ is amenable relative to $qNq$ inside $q\widetilde{M}q$, for each $i = (X,Y,\delta,t) \in I$, \cite[Theorem 2.1 and the remark following it]{OP07} enables us to find a vector $\xi_i \in \rL^2(\langle q\widetilde{M}q,qNq\rangle)$ in such a way that $\Vert \xi_i\Vert_{2, \Tr} \leq 1$, 
\begin{align*}
|\langle x\xi_i,\xi_i\rangle_{\Tr} - \tau(x)| 
&\leq \delta \quad \text{for all $x \in X\cup\{(\theta_t(y)-y)^*(\theta_t(y)-y)\mid y \in Y\}$}, \\ 
\Vert \theta_t(y)\xi_i - \xi_i \theta_t(y)\Vert_{2, \Tr} 
&\leq \delta \quad 
\text{for all $y \in Y$}. 
\end{align*} 
Observe that $\lim_i \langle x\xi_i, \xi_i\rangle _{\Tr} = \tau(x)$ for all $x \in q\widetilde{M}q$ and $\lim_i \Vert y\xi_i - \xi_i y\Vert_{2, \Tr} = 0$ for all $y \in A$. 

Denote by $\mathcal{K} \subset \rL^2(\langle q\widetilde{M}q,qNq \rangle)$ the closed linear subspace generated by $\{ x \lambda_\gamma e_{qNq}\lambda_\gamma^*\mid x \in qMq, \gamma \in \mathbf{F}_2\}$ and by $e : \rL^2(\langle q\widetilde{M}q,qNq \rangle) \to \mathcal K$ the orthogonal projection. Note that $e \in (qMq)' \cap \mathbf B(\rL^2(q \widetilde M q))$. Thus, the net $\xi'_i := p(1-e)\xi_i$ satisfies $\limsup_i \Vert x\xi'_i\Vert_{2, \Tr} \leq \Vert x \Vert_{2, \tau}$ for all $x \in pMp$ and $\lim_i \Vert a \xi_i' - \xi'_i a\Vert_{2, \Tr} = 0$ for all $a \in A$. 

Suppose that $A \not\preceq_M M_1$ and $A \not\preceq_M M_2$. What we have to show is that $A$ is amenable relative to $B$ inside $M$. By contradiction and proceeding as in the first paragraph of the proof of \cite[Theorem 3.2]{Va13}, we may and do assume that {\em no corner of $A$ is amenable relative to $B$ inside $M$}, that is, $Az$ is not amenable relative to $B$ inside $M$ for any nonzero projection $z \in \mathcal Z(A' \cap pMp)$.

Observe that $\langle q\widetilde{M}q,qNq\rangle = q\langle \widetilde{M},N \rangle q$ with $e_{qNq} = qe_N$ ($= e_N q$), where $\langle \widetilde{M},N\rangle$ is the basic extension of $\widetilde{M}$ by the canonical trace preserving conditional expectation $\rE_N : \widetilde M \to N$ and also that the traces $\Tr$ on $\langle q\widetilde{M}q,qNq\rangle$ and $\widetilde{\Tr}\circ\widehat{\rE}_N$ on $\langle \widetilde{M},N\rangle$ with the dual operator-valued weight $\widehat{\rE}_N$ agree since $\widehat{\rE}_N(qe_N) = q\widehat{\rE}_N(e_N) = q$. (It is then natural to denote the latter trace by the same symbol $\Tr$.) Thus, $\rL^2(\langle q\widetilde{M}q,qNq\rangle)$ can be identified with $q\cdot \rL^2(\langle\widetilde{M},N\rangle)\cdot q$. If we identify $M$ with the $\gamma$th free product component $\lambda_\gamma M\lambda_\gamma^*$, then we have the decomposition $\rL^2(N) = \rL^2(M)\oplus(\rL^2(M)\otimes_B \mathcal{X}\otimes_B \rL^2(M))$ as $M$-$M$-bimodules for some $B$-$B$-bimodule $\mathcal{X}$ (see \cite[\S2]{Ue98a}). Then we see, in the same way as in the proof of \cite[Lemma 4.2]{Io12}, that $\rL^2(\langle q\widetilde{M}q,qNq \rangle)\ominus\mathcal{K}$ is identified, as $qMq$-$qMq$-bimodule, with $q\cdot(\rL^2(M)\otimes_B \mathcal{L})\cdot q \subset \rL^2(M)\otimes_B \mathcal{L}$ for some $B$-$M$-bimodule $\mathcal{L}$. Thus, Lemma \ref{L2} implies that $\lim_i \Vert\xi'_i\Vert_{2, \Tr} = 0$, namely $\lim_i \Vert p\xi_i - ep\xi_i\Vert_{2, \Tr} = 0$.

As in the proof of \cite[Theorem 3.4]{Va13}, we can construct an isometry $U : \rL^2(qMq)\otimes\ell^2(\mathbf{F}_2) \to \rL^2(\langle q\widetilde{M}q,qNq\rangle)$ in such a way that $UU^* = e$ and that $U((x\otimes1)\eta(y\otimes1)) = x(U\eta)y$ for all $x,y \in qMq$ and all $\eta \in \rL^2(qMq)\otimes\ell^2(\mathbf{F}_2)$. Put $\zeta_i := U^* p\xi_i \in p\rL^2(qMq)\otimes\ell^2(\mathbf{F}_2)$ for every $i \in I$. Since  $\rL^2(qMq)\otimes\ell^2(\mathbf{F}_2) \subset \rL^2(q\widetilde{M}q)\otimes\ell^2(\mathbf{F}_2) = (q\otimes1)\cdot (\rL^2(\widetilde{M})\otimes\ell^2(\mathbf{F}_2))\cdot(q\otimes1) \subset \rL^2(\widetilde{M})\otimes\ell^2(\mathbf{F}_2)$, we can follow line by line the rest of the proof of \cite[Theorem 3.4, pp 704--709]{Va13} inside $\rL^2(\widetilde{M})\otimes\ell^2(\mathbf{F}_2)$ with the following remarks: 
\begin{itemize} 
\item[1$^\circ$.] $\rL^2(\widetilde{M})=\rL^2(M)\oplus(\rL^2(M)\otimes_B\mathcal{Y}\otimes_B \rL^2(M))$ as $M$-$M$-bimodules for some $B$-$B$-bimodule $\mathcal{Y}$ and hence $\rL^2(\widetilde{M})\ominus \rL^2(M) = \rL^2(M)\otimes_B\mathcal{L}'$ with some $B$-$M$-bimodule $\mathcal{L}'$. 
\item[2$^\circ$.] A key formula \cite[Lemma 3.2]{Va13} (essentially due to Ioana) holds even in the semifinite setting (whose proof goes along that of \cite[Lemma 3.5]{BHR12}). 
\item[3$^\circ$.] The semifinite counterpart of \cite[Theorem 3.1]{Va13} (that is essentially due to Ioana--Peterson--Popa \cite{IPP05}) was already provided by Boutonnet--Houdayer--Raum \cite[Theorem 3.3]{BHR12} and we need to use it in place of \cite[Theorem 3.1]{Va13}.  
\end{itemize}
Following line by line the proof of \cite[Theorem 3.4, pp 704--709]{Va13}, we can then reach a contradiction. Giving the full details is just a task of understanding Vaes's argument modulo the above three remarks. 
\end{proof}

Let $p \in qMq$ be any nonzero projection and $A \subset pMp$ any von Neumann subalgebra. Assume that $A$ is amenable relative to $M_i$ inside $M$ for some $i \in \{1,2\}$. Then, by checking Theorem \ref{T1} (4) and regarding $\rL^2(M)$ as a subspace of $\rL^2(\widetilde{M})$ naturally, we see that $A$ is amenable relative to $M_i$ inside $\widetilde{M}$. For every $t \in (0,1)$, $\theta_t(A)$ is amenable relative to $\theta_t(M_i) = u_i^t M_i u_i^t{}^*$ inside $\widetilde{M}$. The Jones projection $e_{\theta_t(M_i)}$ coincides with $u_i^t e_{M_i} u_i^t{}^*$ so that $\langle \widetilde{M},\theta_t(M_i)\rangle = \langle\widetilde{M},M_i\rangle$ and hence $\theta_t(A)$ is amenable relative to $M_i$ and also to $N$ since $M_i \subset N$. Since $\langle q\widetilde{M}q, qNq\rangle = q\langle \widetilde{M},N\rangle q$ (see the proof of Theorem \ref{T3}), $\theta_t(A)$ is amenable relative to $qNq$ inside $q\widetilde{M}q$ thanks to Theorem \ref{T1} (1). Applying Popa--Vaes's dichotomy result \cite[Theorem 1.6 and Remark 6.3]{PV11} to $\theta_t(A) \subset q\widetilde{M}q = qNq\rtimes_{\Ad(\lambda)}\mathbf{F}_2$, we have that at least one of the following conditions holds: $\theta_t(A) \preceq_{q\widetilde{M}q} qNq$ or $\theta_t(\mathcal{N}_{pMp}(A)\dpr)$ ($\subset \mathcal{N}_{\theta_t(p)\widetilde{M}\theta_t(p)}(\theta_t(A))\dpr$) is amenable relative to $qNq$ inside $q\widetilde{M}q$. Since this is true for every $t \in (0, 1)$, at least one of the following conditions holds: 
\begin{itemize}
\item[(A)] $\theta_t(A) \preceq_{q\widetilde{M}q} qNq$, and hence $\theta_t(A) \preceq_{\widetilde{M}} N$ for some $t \in (0, 1)$ or 
\item[(B)] $\theta_t(\mathcal{N}_{pMp}(A)\dpr)$ is amenable relative to $qNq$ inside $q\widetilde{M}q$ for every $t \in (0, 1)$. 
\end{itemize}

In case (A), we use \cite[Theorem 3.4]{BHR12} (whose proof actually works even when the projection $p$ there lies in $\Proj(\mathcal{M})$ rather than $\Proj(\mathcal{B})$) and the consequence is that $A \preceq_M B$ or $\mathcal{N}_{pMp}(A)\dpr \preceq_M M_i$ for some $i \in \{1,2\}$. In case (B), Theorem \ref{T3} implies that $\mathcal{N}_{pMp}(A)\dpr \preceq_M M_i$ for some $i \in \{1,2\}$ or $\mathcal{N}_{pMp}(A)\dpr$ is amenable relative to $B$ inside $M$. Consequently, we obtain the following result.

\begin{thm}\label{theorem-appendix-AFP} Let $p \in qMq$ be any nonzero projection and $A \subset pMp$  any von Neumann subalgebra. Assume that $A$ is amenable relative to one of the $M_i$ inside $M$. Then at least one of the following holds: 
\begin{itemize} 
\item $A \preceq_M B$. 
\item Either $\mathcal{N}_{pMp}(A)\dpr \preceq_M M_1$ or $\mathcal{N}_{pMp}(A)\dpr \preceq_M M_2$ holds. 
\item $\mathcal{N}_{pMp}(A)\dpr$ is amenable relative to $B$ inside $M$. 
\end{itemize}
\end{thm}

Suppose that $A$ is amenable. Then $A$ is amenable relative to any von Neumann subalgebra with expectation inside $M$. Hence the above dichotomy holds. Suppose moreover that $B$ is also amenable but $\mathcal{N}_{pMp}(A)\dpr$ is {\em not}. Then, it is impossible that $\mathcal{N}_{pMp}(A)\dpr$ is amenable relative to $B$ inside $M$. In fact, there exists a (non-normal) conditional expectation from $\mathbf B(p\rL^2(M))$ onto $p\langle M,B\rangle p$ since $B$ is amenable and thus $\mathcal{N}_{pMp}(A)\dpr$ must be amenable, a contradiction. Therefore, the dichotomy becomes that one of $A \preceq_M B$, $\mathcal{N}_{pMp}(A)\dpr \preceq_M M_1$ and $\mathcal{N}_{pMp}(A)\dpr \preceq_M M_2$ holds true.

\bibliographystyle{plain}

\end{document}